\documentclass{amsart}
\pdfoutput=1
\usepackage[utf8]{inputenc}
\usepackage[T1]{fontenc}
\usepackage{lmodern,microtype}
\usepackage{amssymb,mathtools}
\usepackage{subdepth} 
\usepackage{eucal}

\DeclareRobustCommand{\SkipTocEntry}[5]{}

\usepackage{enumitem}
\setlist[enumerate,1]{label=\textup{(\arabic*)}}
\setlist[enumerate,2]{label=\textup{(\alph*)}}

\usepackage[framemethod=TikZ]{mdframed}
\usepackage{tikz}
\usepackage{tikz-3dplot}
\usetikzlibrary{cd} 
\usetikzlibrary{arrows.meta}

\usepackage[pdftitle={Robustness of topological phases on aperiodic lattices},
pdfauthor={Yuezhao Li},
pdfsubject={Mathematics}
]{hyperref}
\usepackage[noabbrev,capitalize]{cleveref}

\usepackage[lite,alphabetic]{amsrefs}

\numberwithin{figure}{section}
\numberwithin{equation}{section}
\theoremstyle{plain}
\newtheorem{theorem}[equation]{Theorem}
\newtheorem{lemma}[equation]{Lemma}
\newtheorem{proposition}[equation]{Proposition}
\newtheorem{deflem}[equation]{Definition and Lemma}
\newtheorem{corollary}[equation]{Corollary}
\theoremstyle{definition}
\newtheorem{definition}[equation]{Definition}
\theoremstyle{remark}
\newtheorem{remark}[equation]{Remark}
\newtheorem{example}[equation]{Example}

\newcommand{\N}{\mathbb{N}}
\newcommand{\Z}{\mathbb{Z}}

\newcommand{\R}{\mathbb{R}}
\newcommand{\C}{\mathbb{C}}
\newcommand{\T}{\mathbb{T}}
\newcommand{\Mat}{\mathbb{M}}
\newcommand{\Cpt}{\mathbb{K}}
\newcommand{\Bdd}{\mathbb{B}}

\newcommand{\K}{\relax\ifmmode\operatorname{K}\else\textup{K}\fi}
\newcommand{\KK}{\relax\ifmmode\operatorname{KK}\else\textup{KK}\fi}
\newcommand{\KR}{\relax\ifmmode\operatorname{KR}\else\textup{KR}\fi}
\newcommand{\KKR}{\relax\ifmmode\operatorname{KKR}\else\textup{KKR}\fi}
\newcommand{\KO}{\relax\ifmmode\operatorname{KO}\else\textup{KO}\fi}
\newcommand{\KKO}{\relax\ifmmode\operatorname{KKO}\else\textup{KKO}\fi}
\newcommand{\st}{\relax\ifmmode{}^*\else{}\textup{*}\fi}
\newcommand{\Cst}{\relax\ifmmode\mathrm{C}^*\else\textup{C*}\fi}
\newcommand{\Cont}{\mathrm{C}}

\newcommand{\Co}{\mathrm{C}_0}
\newcommand{\Cc}{\mathrm{C}_\mathrm{c}}
\newcommand{\Cb}{\mathrm{C}_\mathrm{b}}
\newcommand{\CRoe}{\mathcal{C}_\Roe}
\newcommand{\CstRoe}{\Cst_\Roe}
\newcommand{\CstuRoe}{\Cst_{\mathrm{u,Roe}}}
\newcommand{\Cl}{\ensuremath{\mathrm{C}\ell}}
\newcommand{\CCl}{\ensuremath{\mathbb{C}\ell}}
\newcommand{\defeq}{\mathrel{\vcentcolon=}}
\newcommand{\dom}{\operatorname{dom}} 
\newcommand{\id}{\operatorname{id}} 
\newcommand{\ev}{\operatorname{ev}} 
\newcommand{\supp}{\operatorname{supp}} 
\newcommand{\Ball}{\operatorname{B}}
\newcommand{\Ad}{\operatorname{Ad}}
\newcommand{\Aut}{\operatorname{Aut}}
\newcommand{\Del}{\operatorname{Del}}
\newcommand{\Roe}{\mathrm{Roe}}
\newcommand{\Gpd}{\mathrm{Gpd}}

\newcommand{\bulkcycle}{\prescript{}{d}{\lambda}_{\Omega_0}}
\newcommand{\pos}{{\scriptstyle\mathsf{X}}} 
\DeclarePairedDelimiter{\set}{\lbrace}{\rbrace} 
\DeclarePairedDelimiter{\abs}{\lvert}{\rvert} 
\DeclarePairedDelimiter{\norm}{\lVert}{\rVert} 
\DeclarePairedDelimiter{\ket}{\lvert}{\rangle} 
\DeclarePairedDelimiter{\bra}{\langle}{\rvert} 
\DeclarePairedDelimiter{\braket}{\langle}{\rangle} 
\DeclarePairedDelimiterX{\braketvert}[3]{\langle}{\rangle}{#1\,\delimsize\vert #2\delimsize\vert\,\mathopen{}#3} 

\begin{document}
\title[Robustness of topological~phases]{Robustness of topological~phases
on aperiodic~lattices}
\author{Yuezhao Li}
\address{Mathematical Institute \\
Universiteit Leiden \\
Einsteinweg 55 \\
2333 CA Leiden \\
The Netherlands}
\email{y.li@math.leidenuniv.nl}
\curraddr{Max-Planck-Institut f\"ur Mathematik \\
Vivatsgasse 7 \\
53111 Bonn \\
Germany}
\email{liy@mpim-bonn.mpg.de}

\begin{abstract}
We study the robustness of topological phases on aperiodic lattices by
constructing \st-homomorphisms from the groupoid model to the 
coarse-geometric model of observable C*-algebras. 
These \st-homomorphisms induce maps in K-theory and Kasparov theory.
We show that the strong topological phases
in the groupoid model are detected by position spectral triples. We show
that topological phases coming from stacking along another Delone set
are always weak in the coarse-geometric sense.
\end{abstract}
\maketitle

\tableofcontents

\setcounter{section}{0}
\renewcommand\theHsection{P1.\thesection}
\renewcommand\thesection{\arabic{section}}

\section{Introduction}
In this article, we aim at a new look at the following question:
\begin{quote}
How to understand the robustness of topological phases of a discrete
physical system, described by a generic aperiodic point pattern?
\end{quote}

We begin with some background around it.  Topological insulators are
materials that are insulating in the bulk but nevertheless possess metallic
edges. The current flowing on the edge is usually quite robust under
disorder coming from impurity of the crystal.  Starting from Bellissard,
van~Elst and Schulz-Baldes
\cites{Bellissard-vElst-SBaldes:NCG_of_QHE,Bellissard:K-theory_in_solid_state_physics},
the non-commutative and \Cst-algebraic tools have been intensively applied
to the study of such materials and in particular, leading to an
interpretation of the Kubo formula in the integer quantum Hall effect
(IQHE) as a quantised index pairing.

In physics literature, the model Hamiltonian of a topological insulator is
usually given by a tight-binding, short-range operator supported on a
periodic, square
lattice, invariant under the translation by its unit cell vector. Such a
Hamiltonian \( H \)  (or its spectral projection) belong to a (noncommutative)
torus \( A_\vartheta=\C\rtimes_\vartheta\Z^d \), 
the noncommutativity coming from a 2-cocycle twist \( \vartheta \)  
given by the external magnetic
field. If the Hamiltonian is not translation invariant, then Bellissard
suggested to replace \( A_\vartheta \) by a twisted crossed product \(
\Cont(\Omega)\rtimes_\vartheta\Z^d \), where \( \Omega \) is a compact
space called the hull of \( H \). 
This description is applicable even if the underlying space is no longer
a square lattice, but comes from a tiling subject to some properties, cf.
\cites{Anderson-Putnam:Substitution_tilings,Sadun-Williams:Tilings_Cantor_fiber_bundles}.

There are, however, amorphous materials, like liquid crystals and glass,
that cannot be modelled using such methods. Even if for quasi-crystals such
that the method of
\cites{Anderson-Putnam:Substitution_tilings,Sadun-Williams:Tilings_Cantor_fiber_bundles}
applies, then one has to essentially provide a (non-canonical) \( \Z^d
\)-labelling of the underlying
lattice, under which the Hamiltonian remains short-range
(cf.\cite{Bourne-Prodan:Noncommutative_Chern_numbers}).

These suggest to use more general point sets in
\( \R^d \), modelling the atomic sites of this material, as their
underlying geometric spaces. The dynamics thereon 
are studied by more general actions other than a \( \Z^d \)-translation.
It was explained in
\cite{Bellissard-Herrmann-Zarrouati:Hulls_aperiodic_solids} that 
such point sets should be \emph{Delone sets} (\cref{def:Delone_set}). 
Let \( \Lambda \) be a Delone set, whose points are considered as the sites
in a physical system. How should one model the observable \Cst-algebra \(
A \) from it?

The choice of an observable \Cst-algebra should be aligned with the
following principle: it should be \emph{large} enough to contain all
possible Hamiltonians, and \emph{small} enough to supply useful homotopy
theory (K-theory).  In recent years, there have been two main approaches to
this modelling problem, which provide toolkits to compute invariants of
topological phases:

\begin{itemize}
\item \emph{``dynamical'' approach} 
describes a crossed product \Cst-algebra,
covariant for the \emph{groupoid} actions on the aperiodic point pattern,
then restricts to the dynamical hull of the point pattern.
This gives a groupoid \Cst-algebra.
\item \emph{``coarse-geometric'' approach} 
describes a \Cst-algebra which is stable
under all \emph{short-range}, \emph{locally--finite-rank} perturbations
that do not close the gap. This leads to a Roe \Cst-algebra.
\end{itemize}

The groupoid approach gives an \'etale groupoid \(
\mathcal{G}_\Lambda\rightrightarrows\Omega_0 \) from a Delone set \(
\Lambda \), and yields a ``tight-binding'' groupoid \Cst-algebra \(
\Cst(\mathcal{G}_\Lambda) \) ,
cf.~\cites{Bellissard-Herrmann-Zarrouati:Hulls_aperiodic_solids,Bourne-Prodan:Noncommutative_Chern_numbers,Bourne-Mesland:Topological_phases
,Mesland-Prodan:Interacting}.
This is a generalisation of the well-studied periodic model: if \(
\Lambda=\Z^d \), then \( \Cst(\mathcal{G}_\Lambda)\simeq\Cst(\Z^d) \) is
the group \Cst-algebra of \( \Z^d \).  Here tight-binding means the
following: the regular representation of \( \Cst(\mathcal{G}_\Lambda) \) is
given by a Hilbert \( \Cont(\Omega_0) \)-module, or equivalently, a
continuous field of Hilbert spaces over \( \Omega_0 \), where \( \Omega_0
\) is unit space of the groupoid \( \mathcal{G}_\Lambda \). The fibres of
this continuous field are canonically identified with \( \ell^2(\omega) \),
where \( \omega\in\Omega_0 \) is a Delone set, either as a translated copy
of \( \Lambda \), or as a weak\st-limit of such sets. Thus \(
\Cst(\mathcal{G}_\Lambda) \) consists of families of model Hamiltonians \(
(H_\omega)_{\omega\in\Omega_0} \), where \( H_\omega \) acts on \(
\ell^2(\omega) \), in a covariant way with respect to the groupoid action.
 
The coarse-geometric approach describes a \Cst-algebra which is stable
under all possible short-range perturbations. This gives rise to a (uniform
or non-uniform) Roe \Cst-algebra \( \CstRoe(\Lambda) \) from \( \Lambda \),
cf.~\cites{Kubota:Controlled_topological_phases,Ewert-Meyer:Coarse_geometry}.
We choose to work with the non-uniform Roe \Cst-algebras, whose advantages
over the uniform ones were explained in \cite{Ewert-Meyer:Coarse_geometry}.

In this setup, a Delone set \( \Lambda\subseteq\R^d \) is considered as a
discrete metric space. Then \( L^2(\R^d) \) carries an ample representation
of \( \Co(\Lambda) \), which generates a Roe \Cst-algebra \(
\CstRoe(\Lambda)\). Physically, a short-range Hamiltonian \( H \) has
matrix coeffcients \( H_{x,y}\defeq\braketvert*{x}{H}{y} \) of fast enough
decay, thus can be approximated by controlled operators. Every unit cell
should have finite degrees of freedom, coming from the number of electron
orbits and their spins. Thus the restriction of \( H \) to any finite
region should have finite-rank. Therefore, the Roe \Cst-algebra \(
\CstRoe(\Lambda) \) can be viewed as the universal \Cst-algebra that
contains all such Hamiltonians.

If we model a topological insulator on a Delone set \( \Lambda \) by the
Roe \Cst-algebra \( \CstRoe(\Lambda) \), then its topological phase is
robust in a very strong sense: such topological phases (and their
associated topological invariants) are robust under any
short-range, locally--finite-rank perturbation that preserves the
spectral gap and the symmetry of the system. 
It follows from the K-theory of
(real) Roe \Cst-algebra (cf.~\eqref{eq:K-theory_Roe_Cst-algebras}) 
that such topological phase are completely classified by their topological
invariants, which are \( \Z \)- or \( \Z/2 \)-valued indiced that belong to
the real K-theory groups of \( \R \). 
The groupoid model \( \Cst(\mathcal{G}_\Lambda) \), on the other hand, has
more involved K-theory and numerical invariants. One natural question is
in which sense those topological phases described by the groupoid model 
are still robust. Moreover, do they
lead to interesting invariants that do not occur in the coarse-geometric
model? 

Questions related to this type of robustness are usually answered by
showing that the numerical index is still continuous, even for some more
general perturbations of the Hamiltonian (which does not have to belong to
the observable \Cst-algebra), together with a ``quantisation'' statement
that the range of the numerical index is \( \Z \). This implies that such
perturbations leave the topological invariants unchanged. 

We describe a different approach to this robustness question by
comparing the groupoid model with the coarse-geometric model on \( \Lambda
\) using the regular representation. The topological invariants defined in
\cite{Bourne-Mesland:Topological_phases} factor through this
representation.  Therefore, since \( \CstRoe(\Lambda) \) is ``robust'', one
only needs to understand which topological phases in \(
\Cst(\mathcal{G}_\Lambda) \) may still survive there. To this end, we introduce
a handy definition of position spectral triples
(\cref{def:position_spectral_triple}). Both the groupoid model and the
coarse-geometric model possess such type of spectral triples. In
particular, the position spectral triple \( \xi^\Roe_\omega \) over the Roe
\Cst-algebra \( \CstRoe(\omega) \) induces an isomorphism \(
\KK(\Cl_{d,0},\CstRoe(\omega))\to\Z \)
(\cref{thm:position_spectral_triple_Roe_isomorphism}).  Then we show that
those topological invariants of the groupoid model, given by localising the
``bulk cycle'' \( \bulkcycle \) at \( \omega\in\Omega_0 \), are pullbacks
of the KK-class of \( \xi^\Roe_\omega \) under the comparison
\st-homomorphism \( \Mat_N(\Cst(\mathcal{G}_\Lambda))\to\CstRoe(\omega) \)
(\cref{thm:position_detects_strong_phases}).  Such invariants are therefore
strong in the sense of \cite{Ewert-Meyer:Coarse_geometry}.

As another application, we show in \cref{thm:stacked_phases_are_weak} that
certain numerical invariants on Delone sets, coming from ``stacking''
lower-dimensional topological phases, must be weak, i.e.~unstable under
perturbation. This generalises a result in
\cite{Ewert-Meyer:Coarse_geometry}, which explains why certain topological
phases of the periodic model ought to be weak. More precisely, the authors
of \cite{Ewert-Meyer:Coarse_geometry} compare the
periodic model and the coarse-geometric model using an injective
\st-homomorphism \( \Cst(\Z^d)\hookrightarrow\CstRoe(\Z^d) \).  It vanishes
on all but one \( \Z \)-component of K-theory. In particular, topological
phases coming from ``stacking'' lower-dimensional topological phases will
all vanish in \( \K_*(\CstRoe(\Z^d)) \). This allows us to conclude that the
induced maps \( \K_*(\Cst(\Z^{d}))\to\K_*(\Cst(\Z^{d+1})) \) also vanishes
there. We show that one can replace \( \Z^{d+1} \) by a product Delone set
of the form \( \Lambda\times L \), where \( \Lambda\subseteq\R^p \) and \(
L\subseteq\R^q \) are both Delone sets. Using the same strategy in
\cite{Ewert-Meyer:Coarse_geometry}, we show that stacked topological phases
factor through the Roe \Cst-algebra of a flasque space, which has vanishing
K-theory. This allows us to conclude that such topological phases and their
numerical invariants have to vanish in the target Roe \Cst-algebra.

\subsection{Notation and conventions}
We fix the following symbols and conventions.

\paragraph{\it Dirac notation} 
Let \( \Lambda \) be a discrete set. We
write \( \ket{x} \) for the function in
\( \ell^2(\Lambda) \) which takes \( 1 \) on \( x \) and \( 0 \) elsewhere;
and \( \bra{x} \) for the rank-one operator \( \ell^2(\Lambda)\to\C \) defined by
\( \ket{y}\mapsto\braket*{x,y} \). Let \( T\in\Bdd(\ell^2(\Lambda)) \),
then by \( \braketvert*{x}{T}{y} \) we shall mean the inner product of \(
\ket*{x} \) and \( T\ket*{y} \). The collection 
\( (T_{x,y})_{x,y\in\Lambda} \)  of \(
T_{x,y}\defeq\braketvert{x}{T}{y} \) are called the \emph{matrix elements}
of \( T \). 

\paragraph{\it Tensor products}
Let \( A \) and \( B \) be \( \Z/2
\)-graded ``real''\Cst-algebras. We will write \( A \otimes B \) for their
\emph{graded}, minimal tensor
product. The grading will cause little difference in this article: we will
mostly consider \( \Z/2 \)-graded \Cst-algebras of the form
\( A\otimes\CCl_{p,q} \) or \( A\otimes\Cl_{p,q} \) for a trivial graded
(real or complex) \Cst-algebra \( A \). In such cases, 
the graded tensor product agrees with the ungraded version.

\paragraph{\it Group\textup{(}oid\textup{)} and Roe \Cst-algebras}
We write \( \Cst(\mathcal{G}) \) for the reduced \Cst-algebra of a groupoid
or a group, and \( \CstRoe(X,\mathcal{H}_X) \) for the Roe \Cst-algebra
defined by an ample module \( (X,\mathcal{H}_X) \). This is aligned with
the convention in \cite{Ewert-Meyer:Coarse_geometry}. As a potentially
confusing point, we note that \( \Cst(\Z^d) \) is the \emph{group}
\Cst-algebra of the discrete group \( \Z^d \), and 
\( \CstRoe(\Z^d) \) is the 
\emph{Roe} \Cst-algebra of the discrete metric space \( \Z^d \).   

\section{Topological phases and \K-theory}
In the simplest form, a (non-interacting) quantum system is
described by an unbounded, self-adjoint operator \( H \) (the Hamiltonian)
acting on a complex, separable Hilbert space \( \mathcal{H} \).
A \emph{topological insulator} is a quantum system
which carries certain symmetries, giving rise to \emph{topological
invariants} that distinguish different \emph{topological phases} of this
system.

The links between topological insulators and (real, bivariant)
K-theory has been investigated throughly in the recent decade,
cf.~\cites{Thiang:K-theoretic_classification,Kellendonk:Cst-algebraic_topological_phases,Bourne-Carey-Rennie:NC_framework_topological_insulators}.
In our case of interest, the links are made as follows:
\begin{enumerate}
\item For aperiodic materials, the physically
relevant symmetries are time-reversal symmetry, particle-hole symmetry and
chiral symmetry. These symmetries may be combined and give rise to several
symmetry types of the system. The
topological phase of such a material is represented by the class \(
[H]\in\KK(\Cl_{n,0},A) \) of the Hamiltonian 
in a certain bivariant K-theory (in fact, \( \KO
\)-theory) group of a real \Cst-algebra \( A \) (the observable
real \Cst-algebra), 
cf.~\citelist{\cite{Thiang:K-theoretic_classification}*{Table 1}
  \cite{Bourne-Carey-Rennie:NC_framework_topological_insulators}*{Table 1}
\cite{Kellendonk:Cst-algebraic_topological_phases}*{Section 6}}.

\item The topological invariant of a topological insulator can be computed
as an \emph{index pairing}.  That is, a \emph{Kasparov product} of the form
\[ 
  \KK(\Cl_{n,0},A)\times\KK(A\otimes\Cl_{0,d},\R)\to\KK(\Cl_{n,d},\R)\simeq\KO_{n-d}(\R),
\]
where 
\[ 
  \KO_{n-d}(\R)\simeq\begin{cases}
  \Z &\text{if }n-d\equiv 0,4{\hphantom{,6,7}}\!\mod 8; \\
  \Z/2 &\text{if }n-d\equiv 1, 2{\hphantom{,6,7}}\!\mod 8; \\
  0 &\text{if }n-d\equiv 3, 5, 6, 7\!\mod 8,
\end{cases}
\]
is the receptacle of \( \Z \)- or \( \Z/2 \)-valued indices that can be
measured in physical experiments.
\end{enumerate}

The section is organised as follows. 
We shall first recall in \cref{sec:real_Cst-algebras} the definitions of
``real'' and real \Cst-algebras and their representations. 
Then we study two special cases: the Clifford algebras
in \cref{sec:Clifford_algebras} and the graded ``real'' \Cst-algebra \(
\Cst(\Z^d)\otimes\CCl_{0,d} \) in \cref{sec:rep_Zd}. We recall a few
properties of Kasparov theory in \cref{sec:real_Kasparov_theory} that will
be used later. 

We introduce a class of spectral triples,
called position spectral triples, in \cref{sec:position_spectral_triple}.
We recall in \cref{sec:fundamental_class} 
that the position spectral triple \( \xi_{\Z^d} \) over \(
\Cst(\Z^d)_\R\otimes\Cl_{0,d} \) represents the fundamental class.
Via the Kasparov product, it generates a surjective group homomorphism
\[ 
  [\xi]\colon
  \KK(\Cst(\Z^d)_\R\otimes\Cl_{0,d})\to\KK(\Cl_{d,d},\R)\simeq\Z.
\]

\subsection{``Real'' and real \Cst-algebras}
\label{sec:real_Cst-algebras}
When we consider physical systems with anti-unitary symmetries, like
time-reversal or particle-hole symmetries, then we must represent them as
anti-unitary operators on \( \mathcal{H} \)  which commute or anti-commute
with the Hamiltonian \( H \). 
This turns the Hilbert space \( \mathcal{H} \)  into a ``real'' Hilbert
space \( (\mathcal{H},\Theta) \), and the observable
\Cst-algebra into a ``real'' \Cst-algebra \( (A,\mathfrak{r}) \). The
latter was introduced by Kasparov \cite{Kasparov:K-homology}*{Section 1,
Definition 3}.

We recall the definition of ``real'' and real \Cst-algebras and Hilbert
spaces. A ``real'' \Cst-algebra is sometimes referred to as
a Real \Cst-algebra (cf.~\cite{Bourne-Kellendonk-Rennie:Cayley_transform },
note the upper case R) or a \( \mathrm{C}^{*,\mathrm{r}} \)-algebra
(cf.~\cite{Kellendonk:Cst-algebraic_topological_phases}) in the literature.

\begin{definition}[\cite{Kellendonk:Cst-algebraic_topological_phases}*{Definition
3.7}]\label{def:Real_Cst-algebra_Hilbert_space}
A ``real'' structure on a \( \Z/2 \)-graded \Cst-algebra \( A \) is 
a conjugate-linear,
grading-preserving \st-automorphism \( \mathfrak{r}\colon A\to A \)
of order 2. A ``real''
\Cst-algebra is a \Cst-algebra together with a ``real'' structure on it.

A ``real'' involution on a \( \Z/2 \)-graded Hilbert space \( \mathcal{H}
\) is a conjugate-linear, grading-preserving automorphism \(
\Theta\colon\mathcal{H}\to\mathcal{H}
\)
of order two. A ``real'' Hilbert space is a Hilbert space \( \mathcal{H} \)
together with a ``real'' involution \( \Theta \)  on it.

A representation of a \( \Z/2 \)-graded ``real'' \Cst-algebra \( A \) is a
\st-homomorphism \( \pi\colon A\to\Bdd(\mathcal{H}) \) for a \( \Z/2
\)-graded ``real'' Hilbert space \( \mathcal{H} \), which intertwines both
the \( \Z/2 \)-gradings and the ``real'' structures.

\end{definition}

Let \( \Theta \) be a ``real'' involution on a \( \Z/2 \)-graded
Hilbert space. Then \( \Bdd(\mathcal{H}) \) becomes a \( \Z/2 \)-graded
``real'' \Cst-algebra with ``real'' structure
\( T\mapsto \Theta T\Theta^* \). A (possibly unbounded) operator on \(
\mathcal{H} \) is said to be ``real'' if it commutes with \( \Theta \). 

Now we recall real \Cst-algebras and real Hilbert spaces.

\begin{definition}\label{def:real_Cst-algebra_Hilbert_space}
A real Hilbert space is a Hilbert space over \( \R \). A real \Cst-algebra
is a norm-closed subalgebra of \( \Bdd(\mathcal{H_\R}) \), where \(
\mathcal{H_\R} \) is a Hilbert space over \( \R \). 
\end{definition}

If \( (A,\mathfrak{r}) \) is a ``real'' \Cst-algebra, then the fixed points
of its real structure
\[ 
A^{\mathfrak{r}}\defeq\set*{a\in A\;\middle|\;\mathfrak{r}(a)=a}
\]
is a real \Cst-algebra. Conversely, if \( A_\R \) is a real \Cst-algebra,
then \( A_\R\otimes_\R\C \) is a ``real'' \Cst-algebra with ``real''
structure
\[ 
\mathfrak{r}(a\otimes_\R z)\defeq a\otimes_\R\overline{z}.
\]

Let \( (X,\tau) \) be an locally compact, Hausdorff involutive space, 
that is, a locally compact Hausdorff space \( X \)
together with a homeomorphism \( \tau\colon X\to X \) satisfying \(
\tau^2=\id_X \). If \( X \) is a manifold, then we also say \( (X,\tau) \)
is a ``real'' manifold. The \Cst-algebra \( \Co(X) \) carries a ``real''
structure
\[ 
    \mathfrak{r}(f)(x)\defeq\overline{f(\tau(x))},
\]
yielding a ``real'' \Cst-algebra \( (\Co(X),\mathfrak{r}) \) as well as
its corresponding real \Cst-algebra \( \Co(X)^\mathfrak{r} \). 
The ``real'' Gelfand--Naimark theorem due to Arens and Kaplansky
\cite{Arens_Kaplansky:Topological_representation}*{Theorem 9.1} asserts
that every commutative ``real'' \Cst-algebra is isomorphic to \(
(\Co(X),\mathfrak{r}) \) for some involutive space \( (X,\tau) \); and
every commutative real \Cst-algebra is isomorphic to 
\( \Co(X)^\mathfrak{r} \).  

\subsubsection{Clifford algebras}
\label{sec:Clifford_algebras}
Let \( \CCl_{p,q} \) be the finite-dimensional \( \Z/2 \)-graded ``real''
\Cst-algebra
generated by \( \gamma^1,\dots,\gamma^p,\rho^1,\dots,\rho^q \), satisfying:
\begin{itemize}
\item \( \gamma^1,\dots,\gamma^p \) are odd, self-adjoint, involutive and 
real;
\item \( \rho^1,\dots,\rho^q \) are odd,
anti-self-adjoint, anti-involutive and real;
\item \( \gamma^1,\dots,\gamma^p,\dots,\rho^1,\dots,\rho^q \) mutually
anti-commute.
\end{itemize}

That is, we require that for all \( i,j \): 
\begin{alignat*}{3}
&(\gamma^j)^2=1,\quad&&(\gamma^j)^*=\gamma^j,\quad&&\mathfrak{r}(\gamma^j)=\gamma^j;
\\
&(\rho^j)^2=-1,\quad&&(\rho^j)^*=-\rho^j,\quad&&\mathfrak{r}(\rho^j)=\rho^j.
\end{alignat*}

The real subalgebra of \( \CCl_{p,q} \) is the \( \R \)-algebra generated by the
same generators and relations. We write \(
\Cl_{p,q}\defeq(\CCl_{p,q})^\mathfrak{r}
\) for this \( \Z/2 \)-graded, real \Cst-algebra. 
Up to Morita equivalence of graded real
\Cst-algebras, there are only eight possible \( \Cl_{p,q} \)
satisfying
\[ 
    \Cl_{p,q}\otimes\Cl_{p',q'}\simeq\Cl_{p+p',q+q'}.
\]
Moreover, \( \Cl_{1,1} \) is isomorphic to the \( \Z/2 \)-graded 
real \Cst-algebra \( \Mat_2(\R)\), whose grading is given by 
diagonal--off-diagonal elements. Thus
\[
    \Cl_{d,d}\simeq\Mat_2(\R)\otimes\Mat_2(\R)\otimes\dots\otimes\Mat_2(\R)
    \simeq\Mat_{2^d}(\R).
\]

Kasparov has constructed a canonical representation of \( \CCl_{p,q} \)
in \cite{Kasparov:K-functor_extension}. Let \( \C^d \) be the
``real'' Hilbert space with basis \( e_1,\dots,e_d \) and ``real''
involution
\[
    \sum_{i=1}^dc_ie_i\;\longmapsto\;\sum_{i=1}^d\overline{c_i}e_i.
\]
Let \( \bigwedge\nolimits^*\C^{d} \) be the exterior algebra of \( \C^d \). It is graded by
the subspace of odd or even differential forms
\(
\bigwedge\nolimits^*\C^{d}=\bigwedge\nolimits^{\text{odd}}\C^d\oplus\bigwedge\nolimits^{\text{even}}\C^d
\).
The ``real'' structure on \( \C^d \) extends to
\( \bigwedge\nolimits^*\C^{d} \), that is,
\[
    \sum_{i_ii_2\dots i_k}a_{i_1i_2\dots i_k}e_{i_1}\wedge
    e_{i_2}\wedge\dots\wedge
    e_{i_k}\longmapsto\sum_{i_ii_2\dots
    i_k}\overline{a_{i_1i_2\dots i_k}}e_{i_1}\wedge
    e_{i_2}\wedge\dots\wedge e_{i_k},
\]
turning \( \bigwedge\nolimits^*\C^{d} \) into a \( \Z/2 \)-graded ``real'' Hilbert space.

\begin{deflem}\label{deflem:standard_representation_Cl}
Let \( j\in\{1,\dots,d\} \) and \(
\lambda_j\colon\bigwedge\nolimits^*\C^{d}\to\bigwedge\nolimits^*\C^{d} \)
be the exterior product with \( e_j \), that is, \(
\lambda_j(\omega)=e_j\wedge\omega \). Then its adjoint \(
\lambda_j^*\colon\bigwedge\nolimits^*\C^{d}\to\bigwedge\nolimits^*\C^{d} \)
is contraction with \( e_j \), that is, \(
\lambda_j^*(\omega)=e_j\mathbin{\lrcorner}\omega \). 

As a consequence, there is a representation of the \( \Z/2 \)-graded
``real'' \Cst-algebra \( \CCl_{p,q} \) on \( \bigwedge\nolimits^*\C^{p+q}
\), sending \( \gamma^j \) to \( \lambda_j+\lambda_j^* \) and \( \rho^j \)
to \( \lambda_{p+j}-\lambda_{p+j}^* \). We call it the standard
representation of \( \CCl_{p,q} \). 
\end{deflem}

\begin{proof}
The operators \( \lambda_j \) and their adjoints satisfy
\[ 
\lambda_i\lambda_j+\lambda_j\lambda_i=0,\quad
\lambda_i^*\lambda_j+\lambda_j\lambda_i^*=\braket*{e_i,e_j}.
\]
Moreover, \( \lambda_j \) and \( \lambda_j^* \) are odd and real for all \(
j\). So we have
\[
(\lambda_j+\lambda_j^*)^2=1,\quad (\lambda_j-\lambda_j^*)^2=-1,\quad
(\lambda_i+\lambda_i^*)(\lambda_{p+j}-\lambda_{p+j}^*)=0.
\]
Thus the operators \(
\lambda_1+\lambda_1^*,\dots\lambda_r+\lambda_r^*,\dots,\lambda_{p+1}-\lambda_{p+1}^*,\dots\lambda_{p+q}-\lambda_{p+q}^*
\) generate a copy of \( \CCl_{p+q} \) in \(
\Bdd(\bigwedge\nolimits^*\C^{p+q}) \).   
\end{proof}

\subsubsection{The \( \Z/2 \)-graded ``real'' \Cst-algebra \(
\Cst(\Z^d)\otimes\CCl_{0,d} \)}
\label{sec:rep_Zd}
Next, we describe a representation of the graded ``real'' 
\Cst-algebra \( \Cst(\Z^d)\otimes\CCl_{0,d} \) following
\cite{Ewert-Meyer:Coarse_geometry}*{Section 4}.
The grading on \( \Cst(\Z^d)\otimes\CCl_{0,d} \) is the tensor 
product of the trivial grading on \( \Cst(\Z^d) \) with the 
standard grading on \( \CCl_{0,d} \). 
The ``real'' structure on \( \Cst(\Z^d) \) is given by the pointwise complex
conjugation, that is,
\[
    \mathfrak{r}(f)(n)\defeq \overline{f(n)},\quad
    f\in\Cc(\Z^d),\;n\in\Z^d.
\]
Then the regular representation of the complex \Cst-algebra \( \Cst(\Z^d)
\) extends to a \st-representation on the ``real'' Hilbert space \(
\ell^2(\Z^d) \), whose ``real'' structure is given by pointwise conjugation.
The ``real'' structure on \( \Cst(\Z^d)\otimes\CCl_{0,d} \) is the tensor product
of this ``real'' structure with the standard one on \( \CCl_{0,d} \).

The Fourier transform maps the ``real'' \Cst-algebra \( \Cst(\Z^d) \) to
the ``real'' \Cst-algebra \( \Cont(\T^d) \). Here \( \T^d \) is a ``real''
\( d \)-torus, with involution \( z\mapsto \overline{z} \).

\begin{lemma}[cf.~\cite{Ewert-Meyer:Coarse_geometry}*{Section
4}]\label{lem:real_Fourier_transform}
There is an isomorphism of \( \Z/2 \)-graded, ``real'' Hilbert spaces
\begin{equation}\label{eq:real_Fourier_transform}
U\colon\ell^2(\Z^d)\otimes\bigwedge\nolimits^*\C^{d}\xrightarrow{\sim}L^2\left(\bigwedge\nolimits^*\mathrm{T}^*\T^d\right)
\end{equation}
given by
\[ 
\ket*{k}\otimes e_{i_1}\wedge\dots \wedge 
e_{i_l}\mapsto \frac{z^k}{z_{i_1}\dots
z_{i_l}}\mathrm{d}z_{i_1}\wedge\dots\wedge \mathrm{d}z_{i_l},
\]
where
\[ 
    z^k\defeq z_1^{k_1}z_2^{k_2}\dots z_d^{k_d},\quad
    k=(k_1,\dots,k_d)\in\Z^d.
\]
This induces an isomorphism of \( \Z/2 \)-graded ``real'' \Cst-algebras 
\[
    \Cst(\Z^d)\otimes\CCl_{0,d}\simeq\Cont(\T^d)\otimes\CCl_{0,d}\simeq\Cont(\T^d,\CCl(\T^d))
\]
where \( \CCl(\T^d) \) is the ``real'' Clifford bundle of the ``real''
manifold \( \T^d \). In particular,
under this isomorphism, the canonical representation of \( \CCl_{0,d} \) as
in \cref{deflem:standard_representation_Cl} is
mapped to the Clifford multiplication of \( \CCl(\T^d) \) on 
\( L^2\left(\bigwedge\nolimits^*\mathrm{T}^*\T^d\right) \).
\end{lemma}

We note that the ``real'' structure \( \mathfrak{r} \) on the ``real'' \Cst-algebra \(
\Cont(\T^d) \) induced by the involution \( z\mapsto \overline{z} \)  
is given by
\[ 
  (\mathfrak{r}f)(x)\defeq \overline{f(-x)},
\]
thus the fixed point algebra is
\[ 
  \Cont(\T^d)^\mathfrak{r}\defeq
  \set*{f\in\Cont(\T^d)\;\middle|\;\overline{f(z)}=f(\overline{z})}.
\]
Under the Fourier transform, \( \Cont(\T^d)^\mathfrak{r} \) is identified
with the real group \Cst-algebra \( \Cst(\Z^d)_\R \),
cf.~\cref{sec:fundamental_class}.

Let \( \pos_j \) be the \( j \)-th position operator on \( \ell^2(\Z^d) \)
defined by
\[ 
    \pos_j\phi(x)\defeq x_j \phi(x),\quad
    \phi\in\Cc(\Z^d),\;x=(x_1,\dots,x_d)\in\Z^d.
\]
Thus the operator \( \pos=\sum_{j=1}^d\pos_j\otimes\gamma^j \) is
an odd, essentially
self-adjoint operator on the \( \Z/2 \)-graded ``real'' Hilbert space \(
\ell^2(\Z^d)\otimes\bigwedge\nolimits^*\C^{d} \).

\begin{lemma}\label{lem:Dirac_element_position_operator}
The Fourier transform \eqref{eq:real_Fourier_transform} induces a unitary
equivalence between the unbounded operator \( \pos=\pos_j\otimes\gamma^j \)
on \( \ell^2(\Z^d)\otimes\bigwedge^*\C^d \) and the Hodge--de Rham operator
\( \mathrm{d}+\mathrm{d}^* \) on \( L^2(\bigwedge^*\mathrm{T}^*\T^d) \).
Both operators preserve the real subspace \(
\ell^2(\Z^d)_\R\otimes\bigwedge\nolimits^*\R^{d}\simeq
L^2(\bigwedge^*\mathrm{T}^*\T^d)_\R \), and hence are ``real''. 
\end{lemma}

\begin{proof}
By definition, we have
\begin{align*} 
    \pos_j\otimes\lambda_j\left(\ket*{k}\otimes e_{i_1}\wedge\dots\wedge
    e_{i_l}\right)&=k_j\ket*{k}\otimes e_{j}\wedge e_{i_1}\wedge\dots\wedge
    e_{i_l} \\
    &=U^*\left(k_j z^k\cdot
    \frac{\mathrm{d}z_j}{z_j}\wedge\frac{\mathrm{d}z_{i_1}}{z_{i_1}}\wedge\dots\wedge\frac{\mathrm{d}z_{i_l}}{z_{i_l}}\right).
\end{align*}
The de Rham operator \( \mathrm{d} \) satisfies
\begin{align*} 
\mathrm{d}\left(\frac{z^k}{z_{i_1}\dots
z_{i_l}}\mathrm{d}z_{i_1}\wedge\dots\wedge
\mathrm{d}z_{i_l}\right) 
&=\sum_{j=1}^d z_1^{k_1}\cdots (k_jz_j^{k_j-1}\mathrm{d}z_j)\cdots z_d^{k_d}\wedge
\frac{\mathrm{d}z_{i_1}}{z_{i_1}}\wedge\dots\wedge\frac{\mathrm{d}z_{i_l}}{z_{i_l}}.
\\
&=\sum_{j=1}^dk_jz^k\cdot\frac{\mathrm{d}z_j}{z_j}\wedge
\frac{\mathrm{d}z_{i_1}}{z_{i_1}}\wedge\dots\wedge\frac{\mathrm{d}z_{i_l}}{z_{i_l}}.
\end{align*}
So \( U^*\mathrm{d}U \) acts by \( \sum_{j=1}^d\pos_j\otimes\lambda_j \) on
\( \ell^2(\Z^d)\otimes\bigwedge\nolimits^*\C^{d} \). Therefore, \(
\mathrm{d}+\mathrm{d}^* \) is unitarily equivalent to the operator
\[
    \sum_{j=1}^d\pos_j\otimes(\lambda_j+\lambda_j^*)=\sum_{j=1}^d\pos_j\otimes\gamma^j
\]
by the standard representation of \( \CCl_{d,0} \) in 
\cref{deflem:standard_representation_Cl}. The generators \( \gamma^j \)
are ``real'' in \( \CCl_{d,0} \), hence maps \( \bigwedge\nolimits^*\R^{d}
\) to \( \bigwedge\nolimits^*\R^{d} \). The \( j \)-th position operator \(
\pos_j \) preserves \( \ell^2(\Z^d)_\R \) because \(
\pos_j\ket*{k}=k_j\ket*{k} \) and \( k_j \) is real for any \( k\in\Z^d \).
Therefore the operator \( \pos=\sum_{j=1}^d\pos_j\otimes\gamma^j \)
preserves \( \ell^2(\Z^d)_\R\otimes\bigwedge\nolimits^*\R^{d} \).
\end{proof}

\subsection{Kasparov theory for real \Cst-algebras}
\label{sec:real_Kasparov_theory}
In his seminal work \cite{Kasparov:K-functor_extension}, Kasparov
introduced a bivariant K-theory for \( \Z/2 \)-graded \Cst-algebras.
Elements in his bivariant K-theory group \( \KK(A,B) \) may be viewed as
``generalised homomorphisms'' from \( A \) to \( B \). Kasparov's
constructions are general enough to be transferred to the setting of real
\Cst-algebras
(cf.~\cite{Bourne-Carey-Rennie:NC_framework_topological_insulators}*{Appendix
A}).

For every pair of \( \Z/2 \)-graded real \Cst-algebras \( A \) and \( B \),
there is an abelian group (the so-called Kasparov group) \( \KK(A,B) \). It
generalises the real K-theory of ungraded real
\Cst-algebras (that is, \( \KO \)-theory) in the following way: 
if \( A=\Cl_{n,0} \) and \( B \) is an ungraded real \Cst-algebra, then    
there is an isomorphism
(cf.~\citelist{\cite{Kasparov:K-functor_extension}*{Section 5, Theorem 7} 
\cite{Bourne-Carey-Lesch-Rennie:KO-valued_spectral_flow}*{Lemma 9.1}}):
\begin{equation}\label{eq:ABS_construction}
  \KO_n(B)\simeq\KK(\Cl_{n,0},B).
\end{equation}

Given \(
\Z/2 \)-graded real \Cst-algebras \( A_1,A_2,B_1,B_2,D \), there is a 
natural group homomorphism
\[ 
  \times_D\colon \KK(A_1,B_1\otimes D)\times\KK(A_2\otimes D,B_2)\to\KK(A_1\otimes A_2,B_1\otimes
  B_2)
\]
called the \emph{Kasparov product},
cf.~\cite{Kasparov:K-functor_extension}*{Section 4, Theorem 4}.
We shall need the following formal constructions with the 
Kasparov product, and refer to
\cites{Kasparov:K-functor_extension,Bourne-Carey-Rennie:NC_framework_topological_insulators}
for details.

\subsubsection*{Functoriality} Let \( f\colon A\to B \) and \( g\colon B\to C
\) be \st-homomorphisms between real \Cst-algebras. Then \( f \) gives an
element \( [f]\in\KK(A,B) \) and \( [g]\in\KK(B,C) \). Moreover, we have
\[ 
  [f]\times_B[g]=[g\circ f]\in\KK(A,C).
\]

\subsubsection*{Graded \( \KO_*(\R) \)-module structure of \( \KO_*(B) \)}
Let \( A_1=\Cl_{n,0} \), \( A_2=\Cl_{m,0} \), \( B_1=B_2=D=\R \). 
Then the Kasparov product gives a group homomorphism  
\[ 
  \times_\R\colon\underbrace{\KK(\Cl_{n,0},\R)}_{\simeq\KO_n(\R)}\times\underbrace{\KK(\Cl_{m,0},\R)}_{\simeq\KO_m(\R)}\to\underbrace{\KK(\Cl_{m+n,0},\R)}_{\simeq\KO_{m+n}(\R)},
\]
where the isomorphisms are given as in \eqref{eq:ABS_construction}. This 
turns \( \KO_*(\R)\defeq\bigoplus_{n}\KO_n(\R) \) into a graded
commutative ring.

Now let 
\( A_1=\Cl_{n,0} \), \( A_2=\Cl_{m,0} \), \( B_1=D=\R \) and \( B_2=B \) be
an ungraded real \Cst-algebra. 
Then the Kasparov product gives a group homomorphism  
\[ 
  \times_\R\colon\underbrace{\KK(\Cl_{n,0},\R)}_{\simeq\KO_n(\R)}\times\underbrace{\KK(\Cl_{m,0},B)}_{\simeq\KO_m(B)}\to\underbrace{\KK(\Cl_{m+n,0},B)}_{\simeq\KO_{m+n}(B)},
\]
which turns \( \KO_*(B)\defeq\bigoplus_{n}\KO_n(B) \) into a graded
commutative module over \( \KO_*(\R) \).

\subsubsection*{Index pairing}
Let \( A_1=\Cl_{n,0} \), \( A_2=\Cl_{0,d} \), \( B_1=B_2=\R \), then
the Kasparov product gives a group homomorphism
\[ 
\times_D\colon\underbrace{\KK(\Cl_{n,0},D)}_{\simeq\KO_n(D)}\times\KK(D\otimes\Cl_{0,d},\R)\to\underbrace{\KK(\Cl_{n,d},\R)}_{\simeq\KO_{n-d}(\R)}.
\]
Thus every element \( \alpha\in\KK(D\otimes\Cl_{0,d},\R) \) induces a group
homomorphism
\begin{equation}\label{eq:index_pairing}
  \times_D\alpha\colon\KO_n(D)\to\KO_{n-d}(\R).
\end{equation}
Moreover, since the Kasparov product is functorial, \( \alpha \) actually
induces an
\( \KO_*(\R) \)-module homomorphism 
\( \KO_*(D)\to\KO_{*-d}(\R) \).

The group homomorphism \eqref{eq:index_pairing}
induced by an element \( \alpha\in\KK(D\otimes\Cl_{0,d},\R) \)
is called an \emph{index pairing}. Such an element \( \alpha \) 
can be represented by a real spectral triple. 

\begin{definition}
Let \( A \) be a \( \Z/2 \)-graded real \Cst-algebra. A real spectral
triple \( (\mathcal{A},\mathcal{H},D) \) over \( A \)  consists of: 
\begin{itemize}
\item A \( \Z/2 \)-graded real Hilbert space \( \mathcal{H} \), together
with a grading-preserving \st-representation 
\( \varphi\colon A\to\Bdd(\mathcal{H}) \); 
\item A dense \st-subalgebra \( \mathcal{A}\subseteq A \);
\item An unbounded, self-adjoint odd operator
\( D\colon\dom D\subseteq \mathcal{H}\to\mathcal{H} \);
\end{itemize}
such that:
\begin{itemize}
\item \( \varphi(a)(1+D^2)^{-1} \) is compact for all \( a\in\mathcal{A} \);
\item For every \( a\in\mathcal{A} \), \( \varphi(a) \) 
maps \( \dom D \) into \( \dom D \); and the graded commutator
\( [D,\varphi(a)] \) extends to a bounded operator on \( \mathcal{H} \).
\end{itemize}
\end{definition}

Let \( \xi\defeq (\mathcal{B},\mathcal{H},D) \) be a 
real spectral triple over a real
\Cst-algebra \( B \), where \( B \) is represented on \( \mathcal{H} \) via
\( \varphi\colon B\to\Bdd(\mathcal{H}) \). 
Then \( \xi \) represents a class \( [\xi]\in\KK(B,\R) \).  

Let \( f\colon A\to B \) be a
\st-homomorphism. Then the \st-algebra \( \mathcal{A}\defeq
f^{-1}(\mathcal{B}) \) is dense in \( A \), and   
\( f^*\xi\defeq (\mathcal{A},\mathcal{H},D) \) is a real spectral triple
over \( A \), where \( A \) is
represented on \( \mathcal{H} \) via \( \varphi\circ f\colon
A\to\Bdd(\mathcal{H}) \).

We call \( f^*\xi\defeq (\mathcal{A},\mathcal{H},D) \) 
in the above example, the
\emph{pullback} spectral triple of \( (\mathcal{B},\mathcal{H},D) \) along
\( f \). Functoriality of Kasparov product states the following: 

\begin{lemma}\label{lem:Kasparov_product_functoriality}
Let \( f\colon A\to B \) be a \st-homomorphism between real 
\Cst-algebras. Let \( \xi \) be a real spectral
triple which represents a class \( [\xi]\in\KK(B\otimes\Cl_{0,d},\R) \).
Then the pullback spectral triple \( f^*\xi \) represents the Kasparov
product \( [f]\times_B[\xi]\in\KK(A\otimes\Cl_{0,d},\R) \). Moreover, the
following diagram commutes\textup{:}
\[ \begin{tikzcd}
\KK(\Cl_{n,0},A) \arrow[rr, "{\times_A[f^*\xi]}"] \arrow[rd, "{\times_A[f]}"'] && \KK(\Cl_{n,d},\R) \\
& \KK(\Cl_{n,0},B) \arrow[ru, "{\times_B[\xi]}"'] 
\end{tikzcd} \]
where all arrows are given by taking Kasparov products.
\end{lemma} 

\subsection{Position spectral triples}
\label{sec:position_spectral_triple}
In the following, we focus on a class of spectral triples, which we
call \emph{position spectral triples} 
as they are built from the position
operators on the corresponding Hilbert space of a 
discrete set \( \Lambda \).
 
\begin{definition}\label{def:position_spectral_triple}
Let \( \Lambda\subseteq\R^d \) be a countable discrete subset.
A \emph{position spectral triple} associated to \( \Lambda \) 
is a real spectral triple of the form
\begin{equation}\label{eq:position_spectral_triple}
\xi_\Lambda\defeq\left(\mathcal{A}\otimes\Cl_{0,d},\quad\ell^2(\Lambda)\otimes\mathcal{K}\otimes\bigwedge\nolimits^*\R^{d},\quad\pos\defeq
\sum_{j=1}^d\pos_j\otimes\id_\mathcal{K}\otimes\gamma^j\right),
\end{equation}
where:
\begin{itemize}
\item \( A \) is a real \Cst-algebra, 
which carries a \st-representation \( \varphi\colon
A\to\Bdd(\ell^2(\Lambda)\otimes\mathcal{K}) \) for some separable real
Hilbert space \( \mathcal{K} \); \( \mathcal{A}\subseteq A \) is a dense
\st-subalgebra; \item \( A\otimes\Cl_{0,d} \) is represented on \(
\ell^2(\Lambda)\otimes\mathcal{K}\otimes\bigwedge\nolimits^*\R^{d} \)
through the tensor product of \( \varphi \) and the standard representation
of \( \Cl_{0,d} \);   
\item \( \pos_j \) is the \( j \)-th position operator on \(
\ell^2(\Lambda) \), given by
\[ 
    (\pos_j\phi)(x)\defeq x_j\phi(x),\qquad \phi\in\Cc(\Lambda),\;
    x=(x_1,\dots,x_d)\in\Lambda\subseteq\R^d;
\]
\item \( \gamma_1,\dots,\gamma_d \) are the generators of \( \Cl_{d,0} \),
represented on \( \bigwedge\nolimits^*\R^{d} \) via the standard
representation. 
\end{itemize}
The condition \( \xi_\Lambda \) being a real spectral triple says that:
\begin{itemize}
\item \(
\varphi(a)(1+\pos^2)^{-1}\in\Cpt(\ell^2(\Lambda)\otimes\mathcal{K}\otimes\bigwedge\nolimits^*\R^{d})
\) for all \( a\in\mathcal{A} \);
\item For every \( a\in\mathcal{A} \), \( \varphi(a) \) maps \( \dom\pos \)
into \( \dom\pos \), and \( [\pos,\varphi(a)] \) extends to a bounded
operator on \(
\ell^2(\Lambda)\otimes\mathcal{K}\otimes\bigwedge\nolimits^*\R^{d} \).  
\end{itemize}
\end{definition}

We note the occurence of an 
auxiliary separable Hilbert space \( \mathcal{K} \). 
The physical meaning of \( \mathcal{K} \) is given in
\cref{rmk:fundamental_domain} for \( \Lambda \)   
a Delone set (cf.~\cref{def:Delone_set}).

The author thanks the anonymous referee for pointing out the following
result:

\begin{lemma}\label{lem:automatic_local_compactness}
Every position spectral triple is \emph{locally compact}. That is, the
matrix elements
\[ 
  \varphi(a)_{x,y}\defeq \braketvert{x}{\varphi(a)}{y}
\]
are compact for all \( a\in \mathcal{A} \) and \( x,y\in\Lambda \).   
\end{lemma}

\begin{proof}
By assumption of a real spectral triple, we have that
\[ 
  \varphi(a)\left(1+\left(\sum_{j=1}^d\pos_j\otimes\id_\mathcal{K}\otimes\gamma^j\right)^2\right)^{-1}\in\Cpt\left(\ell^2(\Lambda)\otimes\mathcal{K}\otimes\bigwedge\nolimits^*\R^{d}\right)
\]
for all \( a\in \mathcal{A} \). Therefore,
\[ 
 \braketvert*{x}{
 \varphi(a)\left(1+\left(\sum_{j=1}^d\pos_j\otimes\id_\mathcal{K}\otimes\gamma^j\right)^2\right)^{-1}}{y}=\left(1+\sum_{j=1}^dy_j^2\right)^{-1}\varphi(a)_{x,y}\otimes\id_{\bigwedge\nolimits^*\R^{d}}
\]
is a compact operator on \( \mathcal{K}\otimes\bigwedge\nolimits^*\R^{d} \). This forces \(
\varphi(a)_{x,y}\in\Cpt(\mathcal{K}) \). 
\end{proof}

\subsubsection{The fundamental class of \( \Cst(\Z^d)_\R\otimes\Cl_{0,d} \)}
\label{sec:fundamental_class}
The last goal of this section is to construct a position spectral triple \(
\xi_{\Z^d,N} \) over \( \Cst(\Z^d)_\R\otimes\Mat_N(\R)\otimes\Cl_{0,d} \),
where \( \Cst(\Z^d)_\R \) is the
real group \Cst-algebra of \( \Z^d \). Moreover, we shall show that taking
the Kasparov product with the
class \( [\xi_{\Z^d,N}]\in\KK(\Cst(\Z^d)_\R\otimes\Cl_{0,d},\R) \)
gives a surjective homomorphism
\[ 
  \KK(\Cl_{d,0},\Cst(\Z^d)_\R)\to\KK(\Cl_{d,d},\R)\simeq\Z.
\]

Let \( \Cc(\Z^d)_\R \) be the real convolution 
algebra of \( \Z^d \), whose elements are
real-valued, finitely supported functions \( f\colon\Z^d\to\R
\), and carries the following convolution product and involution:
\[ 
    (f*g)(x)\defeq\sum_{x_1+x_2=x}f(x_1)g(x_2),\quad
    f^*(x)\defeq f(-x).
\]

The left multiplication of \( \Cc(\Z^d)_\R \) on itself 
extends to a injective \st-representation
\[
  \lambda\colon\Cc(\Z^d)_\R\to\Bdd(\ell^2(\Z^d,\R)).
\]
The closure of \( \lambda(\Cc(\Z^d)_\R) \) in the real 
\Cst-algebra \( \Bdd(\ell^2(\Z^d,\R)) \) is the real group \Cst-algebra 
\( \Cst(\Z^d)_\R \).

Fix a natural number \( N \) and extend \( \lambda \) entrywise to   
\( \Cst(\Z^d)_\R\otimes\Mat_N(\R) \). This gives a representation
\[
  \lambda\otimes\id_N\colon\Cst(\Z^d)_\R\otimes\Mat_N(\R)\to\Bdd(\ell^2(\Z^d,\R)\otimes\R^N)\simeq\Bdd(\ell^2(\Z^d,\R^N)).
\]
To get a position spectral triple over \( \Mat_N(\Cst(\Z^d)_\R) \), we need
another separable real Hilbert space \( \mathcal{K} \). 
Let \( \R^N\hookrightarrow \mathcal{K} \) be any isometric
embedding. Then it induces a corner embedding 
\( e_N\colon \Mat_N(\R)\hookrightarrow
\Cpt(\mathcal{K})\) and gives an invertible element in \(
\KK(\Mat_N(\R),\R) \). We represent \(
\Cst(\Z^d)_\R\otimes\Mat_N(\R) \) on \(
\ell^2(\Z^d)\otimes\mathcal{K} \) via
\( \lambda\otimes e_N \). This gives a position spectral triple 
\begin{equation}\label{eq:spectral_triple_group_Z^d}
    \xi_{\Z^d,N}\defeq\left(\Mat_N(\Cc(\Z^d)_\R)\otimes\Cl_{0,d},\quad
    \ell^2(\Z^d)\otimes\mathcal{K}\otimes\bigwedge\nolimits^*\R^{d},\quad
    \sum_{j=1}^d\pos_j\otimes \id_\mathcal{K}\otimes\gamma^j\right).
\end{equation}
which represents a class \(
[\xi_{\Z^d,N}]\in\KK(\Mat_N(\Cst(\Z^d)_\R)\otimes\Cl_{0,d},\R) \). 
This is the same class for each \( N \) (hence also for \( N=1 \)) because
\( [e_N]\in\KK(\Mat_N(\R),\R)) \) is invertible. 

\begin{theorem}\label{thm:fundamental_class_Zd}
Taking the Kasparov product with \( [\xi_{\Z^d,N}] \), 
\[ 
  \times_{\Mat_N(\Cst(\Z^d)_\R)}[\xi_{\Z^d,N}]\colon
  \KK(\Cl_{d,0},\Mat_N(\Cst(\Z^d)_\R))\to\KK(\Cl_{d,d},\R)\simeq\Z
\]
is a surjective group homomorphism.
\end{theorem}

\begin{proof}
The result is well-known, cf.~\cite{Kasparov:K-functor_extension}*{Section
5, Theorem 7}, whilst we need to appeal to a variant of KK-theory for
``real'' \Cst-algebras called KKR-theory (also defined by Kasparov in
\cite{Kasparov:K-functor_extension}).

Let \( (\T^d,\tau\colon z\mapsto\overline{z}) \) be the
``real'' \( d \)-torus. Then Kasparov showed that the  ``real'' spectral
triple
\begin{equation}\label{eq:HdR_spectral_triple}
  \left(\Cont(\T^d,\CCl(\T^d)),\;
  L^2(\bigwedge\nolimits^*\mathrm{T}^*\T^d),\;
  \mathrm{d}+\mathrm{d}^*\right)
\end{equation}
represents the KKR-thereotic
fundamental class in \( \KKR(\Cont(\T^d,\CCl(\T^d)),\C) \). That is, there
exists a class \( \beta\in\KKR(\C,\Cont(\T^d,\CCl(\T^d))) \) (``the Bott
class'') such that
\[ 
  \beta\times_{\Cont(\T^d,\CCl(\T^d))}\alpha=1\in\KKR(\C,\C).
\]
Now we pass from KKR to KK-theory of real \Cst-algebras. By
\cref{lem:real_Fourier_transform}, the Fourier transform
\eqref{eq:real_Fourier_transform} maps 
\( \Cont(\T^d,\CCl(\T^d)) \) to \( \Cst(\Z^d)\otimes\CCl_{0,d} \), where \(
\Cst(\Z^d) \) carries the real structure \(
\mathfrak{r}(f)(x)=\overline{f(x)} \); and the Hodge--de Rham operator \(
\mathrm{d}+\mathrm{d}^* \) to the position operator \(
\pos=\sum_{j=1}^d\pos_j\otimes\gamma^j \), which is a ``real'' operator. 

There is a natural
isomorphism \( \KKR(A,B)\simeq\KK(A^{\mathfrak{r}_A},B^{\mathfrak{r}_B})
\) (cf.~\cite{Bourne-Kellendonk-Rennie:Cayley_transform}*{Section 2}). 
Under this isomorphism, we find that the Hodge--de Rham spectral triple
\eqref{eq:HdR_spectral_triple}, after composing with a rank-one corner 
embedding \( \R\hookrightarrow\Cpt(\mathcal{K}) \),
is mapped to the position spectral triple \(
\xi_{\Z^d,1}\). Hence its Kasparov product with the Bott class in \(
\KK(\R,\Cst(\Z^d)_\R\otimes\Cl_{0,d}) \) equals \(
1\in\Z\simeq\KK(\Cl_{d,d},\R)
\).
\end{proof}

\begin{remark}
The KKR-class of the Hodge--de Rham spectral triple
\eqref{eq:HdR_spectral_triple} is also called the Dirac element
(cf.~\cite{Kasparov:Equivariant_KK-theory}*{Definition--Lemma
4.2}). It is the KK-theoretic analogue of the fundamental class in
cohomology. We refer to
\cite{Lord-Rennie-Varilly:Riemannian_manifolds_in_NCG}*{Section 4.3} for
the notion of fundamental classes in Kasparov theory and noncommutative
geometry.
\end{remark}

\section{Models of aperiodic topological insulators}

\addtocontents{toc}{\SkipTocEntry}
\subsection*{Convention}
From now on, all Hilbert spaces and \Cst-algebras are assumed to be
real. All \Cst-algebras excepted for the Clifford algebras (and their
tensor products with other \Cst-algebras) are assumed to be ungraded.

In this section, 
we shall describe the groupoid model (cf.~\cref{sec:groupoid_model}) and 
Roe \Cst-algebra model (cf.~\cref{sec:Roe_Cst-algebra_model}) 
of topological insulators on aperiodic lattices. 
A short discussion on the \emph{uniform} Roe \Cst-algebras is given in 
\cref{sec:uniform_Roe_Cst-algebra}.

We assume that the aperiodic
lattice is described by a Delone set \( \Lambda\subseteq\R^d \), i.e.~a
uniformly discrete and relatively dense subset of \( \R^d \), defined
as follows.

\begin{definition}\label{def:Delone_set}
Let \( \Lambda\subseteq\R^d \) be a discrete infinite set. Fix \( 0<r<R \). 
Then \( \Lambda \) is called
\begin{itemize}
\item \emph{\( r \)-uniformly discrete}, if 
\( \#(\Ball(x,r)\cap\Lambda)\leq 1 \) for all \( x\in\R^d \);
\item \emph{\( R \)-relatively dense}, if 
\( \#(\Ball(x,R)\cap\Lambda)\geq 1 \) for all \( x\in\R^d \);
\item \emph{\( (r,R) \)-Delone}, if \( \Lambda \) is both \( r \)-uniformly
discrete and \( R \)-relatively dense. 
\end{itemize}
Denote the collection of \( (r,R) \)-Delone sets in \( \R^d \) by
\( \Del_{(r,R)}(\R^d) \). 
\end{definition}

\begin{remark}\label{rmk:fundamental_domain}
Given a Delone set \( \Lambda \), equipped with the subspace metric \(
d(x,y)\defeq\abs*{x-y} \). The Voronoi tiling
(cf.~\cref{fig:Voronoi_tiling}) decomposes \( \R^d \) into
tiles lablled by \( x\in\Lambda \), and thus decomposes the Hilbert
space \( L^2(\R^d) \) into a direct sum:
\begin{equation}\label{eq:decompose_L^2(R^d)}
L^2(\R^d)\simeq\bigoplus_{x\in\Lambda}L^2(T_x),
\end{equation}
where \( T_x \) is (interior of) the tile associated to the point \(
x\in\Lambda \). If \( \Lambda \) is a periodic lattice, then all \( T_x
\)'s are translated copies of the fundamental domain of \( \Lambda \),
where \( \Lambda \) is identified with the discrete, cocompact subgroup of
\( \R^d \) generated by translations by vectors \( x\in\Lambda \).
In this sense, we may regard \( T_x \)'s as a generalised version 
of ``fundamental domains'' of the
dynamics on the aperiodic lattice \( \Lambda \). 

We may further identify all \( L^2(T_x)
\)'s with a fixed, separable Hilbert space \(
\mathcal{K} \). Thus an operator 
\( T\in\Bdd(L^2(\R^d))\simeq\Bdd(\ell^2(\Lambda)\otimes\mathcal{K}) \) 
can be described by its matrix elements \( (T_{x,y})_{x,y\in\Lambda} \),
where
\[ 
T_{x,y}\defeq\braketvert*{x}{T}{y}\in\Bdd(\mathcal{K}).
\]
\end{remark}

\begin{figure}[ht]
\centering
\begin{minipage}[b]{0.4\linewidth}
\centering
\includegraphics[width=\textwidth]{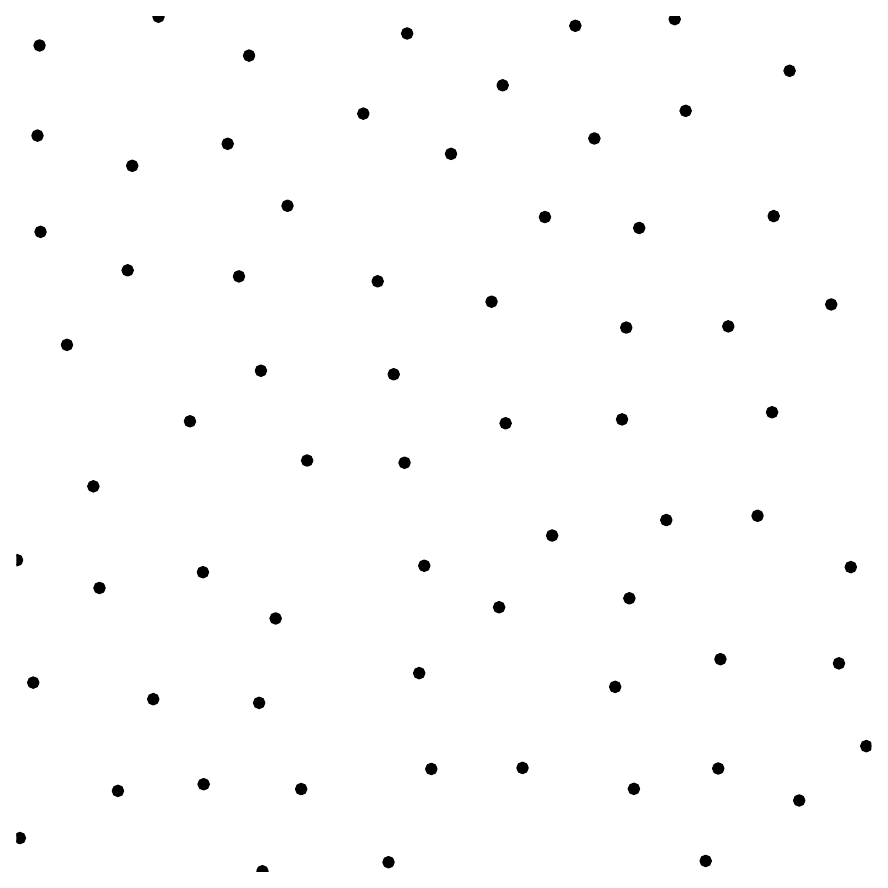}
\end{minipage}
\hspace{1cm}
\begin{minipage}[b]{0.4\linewidth}
\centering
\includegraphics[width=\textwidth]{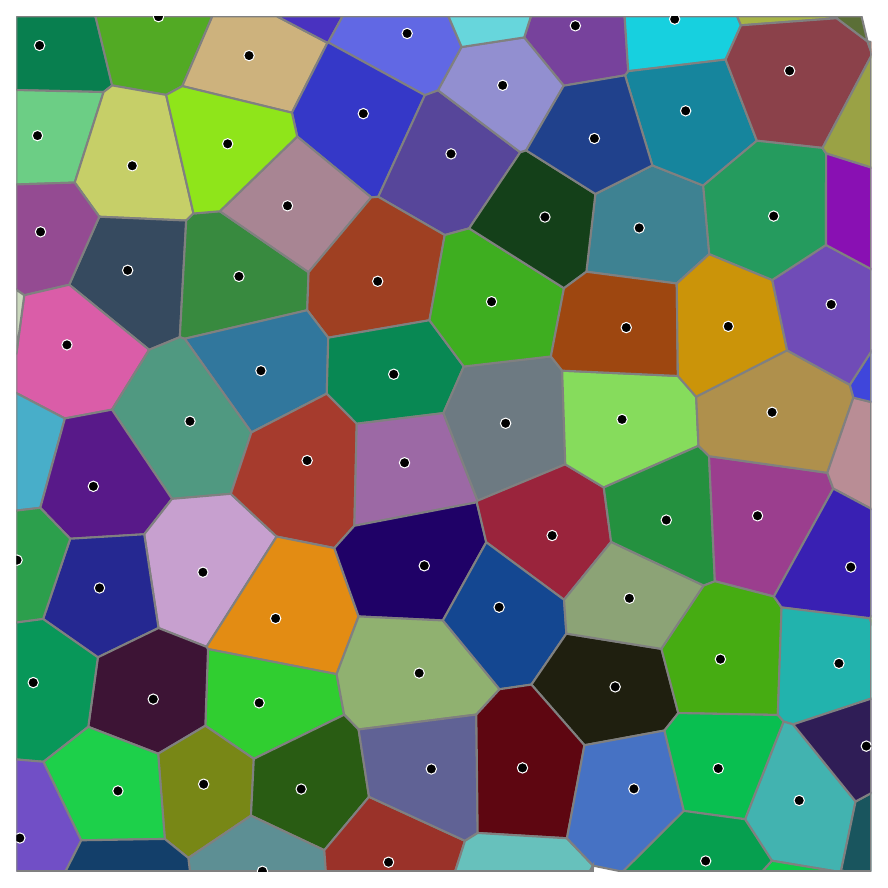}
\end{minipage}
\caption{Voronoi tiling associated to a Delone set}
\label{fig:Voronoi_tiling}
\end{figure}

\subsection{Groupoid \Cst-algebra of a Delone set}
\label{sec:groupoid_model}
In the groupoid approach, a Delone set \( \Lambda \) is considered as
a point in the space of infinite, discrete subsets in \( \R^d \). Its
dynamics is given by translation of the point set by vectors in \( \R^d \).   
More precisely, we identify a Delone set \( \Lambda \) with its
corresponding pure point
measure \( \sum_{x\in\Lambda}\delta_x \) on \( \R^d \), thus identify the
space of all \( (r,R) \)-Delone sets as a subspace of \(
\mathcal{M}(\R^d)=\Cc(\R^d)' \), the space of all measures on \( \R^d \),
equipped with the weak\st-topology. Following
\cite{Bellissard-Herrmann-Zarrouati:Hulls_aperiodic_solids}, we write:
\begin{itemize}
\item \( \mathcal{M}(\R^d) \) for the space of all measures on \( \R^d \);
\item \( \operatorname{QD}(\R^d) \) for the set of all pure point measures
\( \nu \) on \( \R^d \) such that \( \nu(\{x\})\in\N \) for all \( x\in\R^d \);
\item \( \operatorname{UD}_r(\R^d) \) for the subset of \( \operatorname{QD}(\R^d) \)
such that \( \nu(\Ball(x;r))\leq 1 \) for all \( x\in\R^d \).
\end{itemize}

In the context below, we shall not distinguish between a discrete set and
its corresponding pure point measure.

\begin{proposition}[\cite{Bellissard-Herrmann-Zarrouati:Hulls_aperiodic_solids}*{Theorem
1.5, Section 2.1}]\label{prop:spaces_of_measures}
Using the notation above and fix \( 0<r<R \), we have the following\textup{:}
\begin{enumerate}
\item there are inclusions of \emph{closed} sets
\[
    \Del_{(r,R)}(\R^d)\subsetneqq\operatorname{UD}_r(\R^d)\subsetneqq\operatorname{QD}(\R^d)\subsetneqq\mathcal{M}(\R^d)\textup{;}
\]
\item the space \( \operatorname{QD}(\R^d) \) is complete and metrisable\textup{;}
\item the space \( \operatorname{UD}_r(\R^d) \) is compact.
\end{enumerate}
As a corollary of \textup{(1)--(3)}, \( \Del_{(r,R)}(\R^d) \) is a compact
metrisable space.
\end{proposition}

The space \( \Del_{(r,R)}(\R^d) \) carries an action of \( \R^d \)  by
translation. That is, 
given \( \Lambda\in\Del_{(r,R)}(\R^d) \), viewed as a discrete point set in
\( \R^d \); and given \( a\in\R^d \), then \( \Lambda+a \) is the discrete
point set consisting of \( x+a \) for all \( x\in\Lambda \). This coincides
with the \( \R^d \)-action on the space of measures. 

From now on we shall fix a Delone set \( \Lambda \). 
Let \( \Omega_\Lambda \) be the \emph{closure} of the orbit of \( \Lambda
\) under the \( \R^d \)-action. Then \( \Omega_\Lambda \) is a disjoint
union of orbits, hence closed under the \( \R^d \)-action. This allows us
to restrict the topological dynamical system 
\( \Del_{(r,R)}(\R^d)\curvearrowleft\R^d \) to the
smaller space \( \Omega_\Lambda \) and construct the \emph{action groupoid} 
\( \Omega_\Lambda\rtimes\R^d \). The elements are of the form
\( (\omega,a)\in\Omega_{\Lambda}\times\R^d \) and the structure
maps are given by
\begin{alignat*}{3}
r(\omega,a)&=\omega,&&\quad
s(\omega,a)&&=\omega-a,\\
(\omega,a)\cdot(\omega-a,b)&=(\omega,a+b),&&\quad
(\omega,a)^{-1}&&=(\omega-a,-a).
\end{alignat*}
An \emph{abstract transversal} of a topological groupoid 
\( \mathcal{G} \) is a closed subset
\( X\subseteq\mathcal{G}^{0} \) such that \( X \) meets every \( \mathcal{G} \)-orbit
of \( \mathcal{G}^{0} \); and the restrictions of the range and source maps 
to \( \mathcal{G}_X\defeq s^{-1}(X) \) are open,
cf.~\citelist{\cite{Bourne-Mesland:Topological_phases}*{Definition 2.11},
\cite{Muhly-Renault-Williams:Equivalence_groupoids}*{Example 2.7}}. 
Retricting a topological
groupoid \( \mathcal{G} \) to an abstract transversal \( X \) yields a
Morita equivalent groupoid \( \mathcal{G}_X^X\defeq s^{-1}(X)\cap r^{-1}(X) \). 
The Morita equivalence is
implemented by the space \( \mathcal{G}_X \).

\begin{deflem}[\cite{Bourne-Mesland:Topological_phases}*{Proposition 2.14}]
The action groupoid \( \Omega_\Lambda\rtimes\R^d \) admits the following
abstract transversal
\[ \Omega_0\defeq\set{\omega\in\Omega_{\Lambda}\mid 0\in\omega}. \]
Define the groupoid \( \mathcal{G}_\Lambda \) as the restriction of \(
\Omega_{\Lambda}\rtimes\R^d \) to the transversal \( \Omega_0 \).
Then \( \mathcal{G}_\Lambda \) is an \'etale groupoid,
which is Morita equivalent to \( \Omega_{\Lambda}\rtimes\R^d \).
\end{deflem}

\subsubsection{The regular representation}
\label{sec:regular_representation}
The reduced real 
\Cst-algebra of an \'etale groupoid \( \mathcal{G} \) is the
completion of the real groupoid convolution algebra 
\( \Cc(\mathcal{G}) \) under
the regular representation
(cf.~\cite{Khoshkam-Skandalis:Regular_representation}), which we recall
here. The real groupoid convolution algebra \( \Cc(\mathcal{G}) \) consists
of real-valued, compactly supported continuous functions \(
f\colon\mathcal{G}\to\R \), equipped with the convolution product and
involution
\[ 
    f*g(\gamma)\defeq\sum_{\gamma_1\gamma_2=\gamma}f(\gamma_1)g(\gamma_2),\quad
    f^*(\gamma)\defeq f(\gamma^{-1}).
\]
The space \( \Cc(\mathcal{G}) \) is a pre-Hilbert module over \(
\Co(\mathcal{G}^0) \), with right \( \Co(\mathcal{G}^0) \)-module structure 
and inner product (cf.~\cite{Bourne-Mesland:Topological_phases}*{Section 1.3})
\begin{align*} 
(f\cdot \varphi)(\gamma)&\defeq f(\gamma)\varphi(s(\gamma)),\\
\braket*{f,g}(x)&\defeq (f^**g)\big|_{\mathcal{G}^0}(x)=\sum_{\gamma\in
r^{-1}(x)}f(\gamma)g(\gamma),
\end{align*}
for \( \gamma\in\mathcal{G} \), \( x\in\mathcal{G}^0 \), \(
\varphi\in\Co(\mathcal{G}^0) \) and \( f,g\in\Cc(\mathcal{G}) \). Write \(
L^2(\mathcal{G}) \) for the Hilbert \Cst-module completion of \(
\Cc(\mathcal{G}) \). Then the left multiplication
\[ 
    \pi(f)\colon \Cc(\mathcal{G})\to\Cc(\mathcal{G}),\quad g\mapsto f*g
\]
extends to a \st-representation
\[
    \pi\colon\Cc(\mathcal{G})\to\Bdd_{\Co(\mathcal{G}^0)}(L^2(\mathcal{G})),
    \quad f\mapsto\pi(f),
\]
representing \( \Cc(\mathcal{G}) \) by adjointable operators on Hilbert \(
\Co(\mathcal{G}^0) \)-module \( L^2(\mathcal{G}) \).

\begin{definition}
The reduced groupoid \Cst-algebra \( \Cst(\mathcal{G}) \) of an \'etale
groupoid \( \mathcal{G} \) is the completion of \( \Cc(\mathcal{G}) \)
in the norm \( f\mapsto\norm*{\pi(f)} \).
\end{definition}

Let \( x\in\mathcal{G}^0 \). Then the evaluation \st-homomorphism
\[ 
    \ev_x\colon\Co(\mathcal{G}^0)\to\R,\quad \ev_x(\varphi)\defeq\varphi(x)
\]
induces a \st-homomorphism
\[ 
    (\ev_x)_*\colon\Bdd_{\Co(\mathcal{G}^0)}(L^2(\mathcal{G}))\to
    \Bdd(L^2(\mathcal{G})\otimes_{\ev_x}\R),\quad T\mapsto T\otimes \id.
\]
Denote by \( \mathcal{G}_x \) the source fibre of \( \mathcal{G} \) at \( x
\). The Hilbert space
\[
    \mathcal{H}_x\defeq
    L^2(\mathcal{G})\otimes_{\ev_x}\R
\]
is isomorphic to \( \ell^2(\mathcal{G}_x) \), sending \( f\otimes t \)
to the restriction of \( t\cdot f \) to \( \mathcal{G}_x \).
Thus \( \pi_x\defeq(\ev_x)_*\circ\pi \) gives a representation of \(
\Cst(\mathcal{G}) \) on \( \mathcal{H}_x \). We call it the
\emph{localised} representation (of the regular representation \( \pi \))
at \( x\in\mathcal{G}^0 \); and call \( \mathcal{H}_x \) the
\emph{localised} Hilbert space at \( x\in\mathcal{G}^0 \). 

Now we let \( \mathcal{G} \) be 
the groupoid of Delone sets \( \mathcal{G}_\Lambda
\). By definition, its regular representation is defined on the Hilbert
\( \Cont(\Omega_0) \)-module 
\( L^2(\mathcal{G}_\Lambda)_{\Cont(\Omega_0)} \). 
This is a Hilbert module over a unital, commutative \Cst-algebra, and hence
equivalent to a continuous field of Hilbert spaces over \( \Omega_0 \). 
The source fibre of \( \mathcal{G}_\Lambda \) at \( \omega\in\Omega_0 \)   
is given by
\[ 
    \set*{(\omega-x,-x)\;\middle|\;0\in\omega-x}
    =\set*{(\omega-x,-x)\;\middle|\;x\in\omega},
\]
which is in bijection with the Delone set \( \omega \) via \(
(\omega-x,-x)\mapsto x\). Thus the localised Hilbert space \(
\mathcal{H}_\omega \) is unitarily isomorphic to \( \ell^2(\omega) \) via
(cf.~\cite{Bourne-Mesland:Topological_phases}*{Section 4.1}):
\begin{equation}\label{eq:rho_omega}
\begin{aligned}
\rho_\omega\colon\mathcal{H}_\omega\xrightarrow{\sim}\ell^2(s^{-1}(\omega)),&\qquad
\rho_\omega(f\otimes t)(x)\defeq tf(\omega-x,-x).\\
\rho_\omega^{-1}\colon\ell^2(s^{-1}(\omega))\xrightarrow{\sim}\mathcal{H}_\omega,&\qquad\rho_\omega^{-1}(\ket*{x})\defeq\ket*{\omega-x,-x},
\end{aligned}
\end{equation}
where \( \ket*{\omega-x,-x} \) is the equivalence class in \(
\mathcal{H}_\omega \) of any continuous function \( f \)  on \( \mathcal{G} \),
such that \( \supp f\cap s^{-1}(\omega)=\{(\omega-x,-x)\} \) and \(
f(\omega-x,-x)=1 \). 

\begin{lemma}[\cite{Bourne-Mesland:Topological_phases}*{Section 4.1}]
The localised representation \(
\pi_\omega\colon\Cst(\mathcal{G}_\Lambda)\to\Bdd(\ell^2(\omega)) \) is given by the formula
\begin{equation}\label{eq:pi_omega}
(\pi_\omega(f)\psi)(x)\defeq\sum_{y\in\omega}f(\omega-x,y-x)\psi(y)
\end{equation}
for \( f\in\Cst(\mathcal{G}_\Lambda) \), \( \psi\in\ell^2(\omega) \) and \(
x\in\omega \). Therefore,
\[ 
    \braketvert{x}{\pi_\omega(f)}{y}=f(\omega-x,y-x),\quad x,y\in\omega.
\]
\end{lemma}

When we consider systems with a finite number of
internal degrees of freedom inside every
lattice site, e.g.~spins of electrons, 
then we must replace the observable \Cst-algebra
by its matrix algebra \(
\Mat_N\Cst(\mathcal{G}_{\Lambda})
=\Cst(\mathcal{G}_\Lambda)\otimes\Mat_N(\R) \).
In the same way as \cref{sec:position_spectral_triple},
we extend the regular 
representation of \( \Cst(\mathcal{G}_\Lambda) \) to a representation
\[ 
\pi_\omega^N\defeq\pi_\omega\otimes e_N,\quad
\Cst(\mathcal{G}_\Lambda)\otimes\Mat_N(\R)\to\Bdd(\ell^2(\omega)\otimes\mathcal{K}),
\]
where \( \mathcal{K} \) is a separable Hilbert space, and 
\(  e_N\colon\Mat_N(\R)\hookrightarrow\Cpt(\mathcal{K}) \) is any
rank-\( N \) corner embedding induced by an isometry \( \R^N\hookrightarrow
\mathcal{K}\). We may describe the representation 
\( \pi_\omega^N \)
by an infinite matrix indexed by \( x,y\in\omega \), with matrix elements
\begin{equation}\label{eq:pi_omega^N}
\pi^N_\omega(f\otimes S)_{x,y}\defeq\braketvert*{x}{\pi_\omega(f)}{y}\cdot
S\in\Mat_N(\R).
\end{equation}
for all \( f\in\Cst(\mathcal{G}_\Lambda) \) and \( S\in\Mat_N(\R) \). 

\subsubsection{Position spectral triple over the groupoid \Cst-algebra}
In order to construct a spectral triple over
\( \Cst(\mathcal{G}_\Lambda) \) or \(
\Mat_N\Cst(\mathcal{G}_\Lambda) \),
one way is to first construct an unbounded
Kasparov \( \Cst(\mathcal{G}_\Lambda) \)-\( \Cont(\Omega_0) \)--module, then
take its (unbounded) Kasparov product with a \st-homomorphism \(
\Cont(\Omega_0)\to\R \). Such an unbounded Kasparov module was given in
\cite{Bourne-Mesland:Topological_phases}*{Section 2.3.1} using the ``position''
operators on the Hilbert \( \Cont(\Omega_0) \)-module \(
L^2(\mathcal{G}_\Lambda) \):
\begin{equation}\label{eq:bulk_cycle}
\bulkcycle\defeq\left(\Cc(\mathcal{G}_\Lambda)\otimes\Cl_{0,d},\quad
L^2(\mathcal{G}_\Lambda)_{\Cont(\Omega_0)}\otimes\bigwedge\nolimits^*\R^{d},\quad\sum_{i=1}^dX_j\otimes\gamma^j\right),
\end{equation}
where
\[
c_j\colon\mathcal{G}_\Lambda\to\R,\quad 
c_j(\omega,x)\defeq x_j; \qquad
X_jf(\omega,x)\defeq c_j(\omega,x)f(\omega,x).
\]
The map \( c_j\colon\mathcal{G}_\Lambda\to\R \) is 
an \emph{exact} groupoid cocycle in the
sense of \cite{Mesland:Groupoid_cocycles}*{Definition 4.1.2}. It
follows from the general construction in
\cite{Mesland:Groupoid_cocycles}*{Theorem 3.2.2} that \( \bulkcycle \) is
an unbounded Kasparov module. Following \cite{Bourne-Mesland:Topological_phases}, we
call \( \bulkcycle \) the \emph{bulk cycle}.  

Now we construct the position spectral triple by localising \( \bulkcycle
\) at any \( \omega\in\Omega_0 \). As explained in the paragraph after
\cref{def:Delone_set}, the physically relevant Hilbert space is \(
\ell^2(\omega)\otimes\mathcal{K} \) instead of \( \ell^2(\omega) \) or \(
\ell^2(\omega)\otimes\R^N \).  
This is fixed by taking any isometry \( \R^N\hookrightarrow \mathcal{K} \)
and its induced corner embedding \( e_N\colon \Mat_N(\R)\hookrightarrow
\Cpt(\mathcal{K}) \). 

\begin{theorem}\label{thm:position_spectral_triple_groupoid}
Let \( \omega\in\Omega_0 \) and \( \mathcal{K} \) be a Hilbert space. 
The following data
\begin{equation}\label{eq:groupoid_position_spectral_triple}
\xi_{\omega,N}^\Gpd\defeq\left(\Mat_N(\Cc(\mathcal{G}_\Lambda))\otimes\Cl_{0,d},\quad\ell^2(\omega)\otimes\mathcal{K}\otimes\bigwedge\nolimits^*\R^{d},\quad\sum_{j=1}^d\pos_j\otimes
\id_\mathcal{K}\otimes\gamma^j\right)
\end{equation}
is a position spectral triple. It represents a class in
\[
    \KK(\Mat_N\Cst(\mathcal{G}_\Lambda)\otimes\Cl_{0,d},\R)
    \simeq\KK(\Cst(\mathcal{G}_\Lambda)\otimes\Cl_{0,d},\R),
\]
which is the
Kasparov product of the following (unbounded) Kasparov modules and
\st-homomorphisms:
\begin{enumerate}
\item the bulk cycle \( \bulkcycle \), which gives a class in \(
\KK(\Cst(\mathcal{G}_\Lambda)\otimes\Cl_{0,d},\Cont(\Omega_0)) \)\textup{;}
\item the evaluation \st-homomorphism \(
\ev_\omega\colon\Cont(\Omega_0)\to\R \), which gives a class in \(
\KK(\Cont(\Omega_0),\R) \)\textup{;}
\item the corner embedding \( e_N\colon \Mat_N(\R)\hookrightarrow
\Cpt(\mathcal{K}) \), which gives an invertible element in \(
\KK(\Mat_N(\R),\R) \).  
\end{enumerate}
\end{theorem}

\begin{proof}
The composition of (1)--(3) is easy to compute because either (2) or (3)
can be described by a Hilbert module with zero operator. The tensor product
of these Hilbert modules
\[
    \left(L^2(\mathcal{G}_\Lambda)\otimes_{\ev_\omega}\R
    \otimes\mathcal{K}\right)\otimes\bigwedge\nolimits^*\R^{d}
\]
is isomorphic to \(
\ell^2(\omega)\otimes\mathcal{K}\otimes\bigwedge\nolimits^*\R^{d} \) via
the localised representation at \( \omega \).  Thus it suffices to check
that \( c_j \) is mapped to \( \pos_j \) under the composition and that
\eqref{eq:groupoid_position_spectral_triple} is a spectral triple. By
\eqref{eq:rho_omega}, the operator \( X_j\otimes \id_{\mathcal{K}}\otimes
\id_{\bigwedge\nolimits^*\R^{d}} \) acts on \(
\ell^2(\omega)\otimes\mathcal{K}\otimes\bigwedge\nolimits^*\R^{d} \) sends
\( \ket*{x}\otimes v\otimes w \) to \( x_j\ket*{x}\otimes v\otimes w \),
where \( x\in\omega \) is considered an element in \( \R^d \). It is a
spectral triple by \cite{Bourne-Mesland:Topological_phases}*{Proposition
4.1}. Thus \eqref{eq:groupoid_position_spectral_triple} is a position
spectral triple.
\end{proof}

\subsection{Roe \Cst-algebra of a Delone set}
\label{sec:Roe_Cst-algebra_model}
The coarse-geometric approach uses the Roe \Cst-algebra \( \CstRoe(\Lambda)
\) of the Delone set \( \Lambda \).  
It consists of operators on \( \ell^2(\Lambda,\mathcal{K}) \)
which are locally compact and
can be approximated by controlled operators. In this setting, we treat a
Delone set \( \Lambda\subseteq\R^d \) as discrete metric space, equipped with the
subspace topology from \( \R^d \).

\begin{definition}
An \emph{ample \( X \)-module} or \emph{ample module over \( X \)} is given
by a pair \( (X,\mathcal{H}_X) \), where \( X \) is a proper
metric space, and \( \mathcal{H}_X \) is a Hilbert space that carries a 
non-degenerate representation
\( \varrho\colon\Co(X)\to\Bdd(\mathcal{H}_X) \),
which is \emph{ample} in the sense that
\begin{center}
\( \varrho(f)\in\Cpt(\mathcal{H}_X) \)\quad if and only if\quad \( f=0 \).
\end{center}
\end{definition} 

\begin{example}[\cite{Willett-Yu:Higher_index_theory}*{Example
  4.1.2--4.1.4}]\label{ex:standard_ample_module}
Let \( X \) be a Riemannian manifold with volume form \( \mu \). Assume that
every connected component of \( X \) has dimension greater equal than 1, e.g.
\( X=\R^d \) for \( d\geq 1 \) with the Lebesgue measure. Then pointwise
multiplication gives an ample \( X \)-module \( \mathcal{H}_X\defeq L^2(X,\mu) \).

If \( X \) is a discrete metric space with the counting measure, then the
representation \( \Co(X)\to\Bdd(\ell^2(X)) \) by multiplication is no
longer ample. In such cases, we replace \( \ell^2(X) \) by 
by \( \mathcal{H}_X\defeq\ell^2(X,\mathcal{K})=\ell^2(X)\otimes\mathcal{K} \), 
where \( \mathcal{K} \) is any separable Hilbert space, and the
representation maps \( f\in\Co(X) \) to the tensor product of pointwise
multiplication with \( f \) with the identity operator on \( \mathcal{K}
\). Then \( \mathcal{H}_X \) is an ample \( X \)-module. 

We call \( \mathcal{H}_X \) in the above cases, the \emph{standard} ample
\( X \)-module.  
\end{example}

\begin{definition}[\cite{Willett-Yu:Higher_index_theory}*{Definition 4.1.7
and 5.1.1}]\label{def:support_controlled_locally_compact}
Let \( (X,\mathcal{H}_X) \)
and \( (Y,\mathcal{H}_Y) \) be ample modules, carrying 
ample representations \( \varrho_X\colon\Co(X)\to\Bdd(\mathcal{H}_X)
\) and \( \varrho_Y\colon\Co(Y)\to\Bdd(\mathcal{H}_Y) \). 
\begin{itemize}
\item The \emph{support} of an operator 
\( T\in\Bdd(\mathcal{H}_X,\mathcal{H}_Y) \), 
denoted by \( \supp(T) \), is the collection of all points \( (y,x)\in
Y\times X \), such that \( \varrho_Y(\chi_V)T\varrho_X(\chi_U)\neq 0 \)
holds for all open neighbourhood \( U \) of \( x \) and \( V \) of
\( y \).
\item A subset \( \mathcal{E}\subseteq X\times X \) is called
\emph{controlled}, if
\[ 
\sup\set*{\mathrm{d}_X(x_1,x_2)\;\middle|\;(x_1,x_2)\in\mathcal{E}}<+\infty.
\] 
\item An operator \( T\in\Bdd(\mathcal{H}_X) \) is called \emph{locally
compact}, if for any \( f\in\Cc(X) \),
\( T\varrho(f) \) and \( \varrho(f)T \) are compact operators. 
\item An operator \( T\in\Bdd(\mathcal{H}_X) \) is called
\emph{controlled}, if \( \supp(T) \) is controlled.  
\end{itemize}
\end{definition}

\begin{definition}[\cite{Willett-Yu:Higher_index_theory}*{Definition 5.1.4}]
Let \( (X,\mathcal{H}_X )\) be an ample module. 
The Roe \Cst-algebra \( \CstRoe(X,\mathcal{H}_X) \)
is the \Cst-algebra generated by all 
\emph{locally compact}, \emph{controlled} operators on \( \mathcal{H}_X \). 
\end{definition}

\begin{example}[\cite{Ewert-Meyer:Coarse_geometry}*{Example
1}]\label{ex:discrete_space_locally_compact_controlled}
Let \( \Lambda \) be a discrete metric space and 
\( \mathcal{H}_\Lambda=\ell^2(\Lambda,\mathcal{K}) \) be the
standard ample \( \Lambda \)-module. Describe an operator \(
T\in\Bdd(\mathcal{H}_\Lambda) \) 
by an infinite matrix \( (T_{x,y})_{x,y\in \Lambda} \)
where
\[ 
    T_{x,y}\defeq\braketvert*{x}{T}{y}\in\Bdd(\mathcal{K}).
\]
Then the support of \( T \) is given by
\[ 
    \supp(T)=\set*{(x,y)\in X\times X\;\middle|\;T_{x,y}\neq 0}.
\]
The operator \( T \) is locally compact, if and only if \( T_{x,y}\in\Cpt(\mathcal{K}) \) 
for all \( x,y\in \Lambda \); controlled, if and only if there exists \( R>0 \), such that \(
T_{x,y}=0 \) whenever \( \mathrm{d}(x,y)>R \).
\end{example}

A \st-homomorphism between 
Roe \Cst-algebras can be induced by a certain (not necessarily continuous)
map between their underlying spaces. Such maps are called \emph{coarse}. 
Their induced \st-homomorphism are constructed using \emph{covering isometries}.

\begin{definition}[\cite{Willett-Yu:Higher_index_theory}*{Definition
5.1.10}]\label{def:coarse_maps_close}
Let \( (X,\mathrm{d}_X) \) and \( (Y,\mathrm{d}_Y) \) 
be proper metric spaces.
A map \( f\colon X\to Y \)
is called \emph{coarse}, if the following holds:
\begin{enumerate}
\item The \emph{expansion function} \( \omega_f\colon
[0,\infty)\to[0,\infty) \), given by the formula
\[ 
   \omega_f(r)\defeq\sup\set*{\mathrm{d}_Y(f(x_1),f(x_2))
   \;\middle|\;\mathrm{d}_X(x_1,x_2)\leq r}
\]
satisfies \( \omega_f(r)<+\infty \) for all \( r\geq 0 \). 
\item The map \( f \) is proper, i.e.~\( f^{-1}(K)\subseteq X \) is pre-compact for
any compact set \( K\subseteq Y \).  
\end{enumerate}

Two coarse maps \( f,g\colon X\rightrightarrows Y \) are \emph{close}, if
there exists \( c\geq 0 \) such that for all \( x\in X \), \(
d_Y(f(x),g(x))\leq c \).
\end{definition}

\begin{deflem}[\cite{Willett-Yu:Higher_index_theory}*{Definition
5.1.11, Lemma 5.1.12}]\label{deflem:covering_isometry}
Let \( (X,\mathcal{H}_X) \) and \( (Y,\mathcal{H}_Y) \) be ample modules
and let \( f\colon X\to Y \) be a coarse map. A covering isometry for \( f
\) is an isometry \( V\colon\mathcal{H}_X\to\mathcal{H}_Y \) such that
there exists \( t\geq 0 \) such that \( \mathrm{d}(y,f(x))<t \) whenever \(
(y,x)\in\supp(V) \).

The \st-homomorphism
\[
    \Ad_V\colon\Bdd(\mathcal{H}_X)\to\Bdd(\mathcal{H}_Y),\qquad T\mapsto
    VTV^*
\]
restricts to a \st-homomorphism \(
\CstRoe(X,\mathcal{H}_X)\to\CstRoe(Y,\mathcal{H}_Y) \). Its induced maps
\[ 
    f_*\colon\KK(\Cl_{n,0},\CstRoe(X,\mathcal{H}_X))\to\KK(\Cl_{n,0},\CstRoe(Y,\mathcal{H}_Y))
\]
depends only on \( f \) and not on the choice of \( V \).  
\end{deflem}

The \emph{coarse category} consists of proper metric spaces as objects, and
equivalence classes of coarse maps as arrows, where two coarse maps are
equivalent if and only if they are close. 
A \emph{coarse equivalence} is an isomorphism in the
coarse category. If \( X\subseteq Y \) is a metric subspace that is
coarsely equivalent to \( Y \), then we also say that \( X \) is
\emph{coarsely dense} in \( Y \).
In particular, any Delone set \( \Lambda \) of \( \R^d \)
is coarsely dense in \( \R^d \). 

\begin{proposition}[\cite{Willett-Yu:Higher_index_theory}*{Proposition
4.3.5}]\label{prop:Roe_Cst-algebra_independent_module}
Let \( (X,\mathcal{H}_X) \) and \( (Y,\mathcal{H}_Y) \) be ample modules
and \( f\colon X\to Y \) be a coarse equivalence, then there exists a
covering isometry \( V\colon \mathcal{H}_X\to\mathcal{H}_Y\) that is
a unitary equivalence.
\end{proposition}

As a special case, we have the following:
\begin{proposition}\label{prop:Roe_Cst-algebra_coarsely_invariant}
Let \( Y \) be a discrete metric space, and \( X\subseteq Y \) be a
coarsely dense subset. Let \( \mathcal{H}_X\defeq\ell^2(X,\mathcal{K}) \)
and \( \mathcal{H}_Y\defeq\ell^2(Y,\mathcal{K}) \) be their standard ample
modules. Then the operator \( V\colon \mathcal{H}_X\to\mathcal{H}_Y \),
\( V\ket*{x}\defeq\ket*{x} \) is a covering isometry for the isometric
embedding \( \iota\colon X\hookrightarrow Y \) and induce isomorphisms
\[
    \iota_*\colon\KK(\Cl_{n,0},\CstRoe(X,\mathcal{H}_X))
    \to\KK(\Cl_{n,0},\CstRoe(Y,\mathcal{H}_Y)).
\]
\end{proposition}

\begin{proof}
The support of \( V \) is
\[ 
    \supp(V)=\set*{(y,x)\in Y\times X\;\middle|\;\braketvert*{y}{V}{x}\neq
    0}=\set*{(\iota(x),x)\;\middle|\;x\in X}.
\]
So \( \mathrm{d}(y,\iota(x))=0 \) for all \( (y,x)\in\supp(V) \).
Therefore \( V \) is an covering isometry for \( \iota \), and \( \Ad_V \)
induces a \st-homomorphism \(
\iota_*\colon\CstRoe(X,\mathcal{H}_X)\to\CstRoe(Y,\mathcal{H}_Y) \). Since
\( \iota \) is a coarse equivalence, there exists a covering isometry \(
V'\colon\mathcal{H}_X\to\mathcal{H}_Y \) that is a unitary equivalence. 
Hence, \( \Ad_{V'} \) induces an isomorphism
\[
    \KK(\Cl_{n,0},\CstRoe(X,\mathcal{H}_X))\to\KK(\Cl_{n,0},\CstRoe(Y,\mathcal{H}_Y)).
\]
By \cref{deflem:covering_isometry}, the induced maps of \( \Ad_V \)
coincides with that of \( \Ad_{V'} \), hence an isomorphism as well. 
\end{proof}

\begin{remark}
By
\cref{prop:Roe_Cst-algebra_independent_module,prop:Roe_Cst-algebra_coarsely_invariant},
given two
Delone set \( \Lambda_1,\Lambda_2\subseteq \R^d \) and ample modules \(
\mathcal{H}_{\Lambda_1},\mathcal{H}_{\Lambda_2} \), then there exists an
isomorphism \(
\CstRoe(\Lambda_1,\mathcal{H}_{\Lambda_1})\simeq\CstRoe(\Lambda_2,\mathcal{H}_{\Lambda_2})
\). More generally, all of them are isomorphic to the Roe \Cst-algebra of
\( \R^d \) defined by its standard ample module,
cf.~\cite{Ewert-Meyer:Coarse_geometry}*{Theorem 2}. 
However, we note that the
inclusion \( \iota\colon\Lambda\hookrightarrow\R^d \)
does \emph{not} induce an isomorphism (via a covering
isometry) between the corresponding Roe
\Cst-algebras \( \CstRoe(\Lambda,\mathcal{H}_\Lambda) \) and \(
\CstRoe(\R^d,\mathcal{H}_{\R^d}) \) defined using their standard ample
modules. An isomorphism between the Roe \Cst-algebras of 
coarsely equivalent spaces \( X \) and \( Y \) was constructed in
\cite{Ewert-Meyer:Coarse_geometry}*{Theorem 3}. However, these spaces
carry different ample modules, involving an inverse of the coarsely dense
embedding \( \iota\colon X\to X\coprod Y \) in the coarse category.
\end{remark}

As opposed to the groupoid \Cst-algebra \( \Cst(\mathcal{G}_\Lambda) \),
the Roe \Cst-algebra \( \CstRoe(\Lambda) \) has simple
K-theory groups. It follows from coarse Mayer--Vietoris sequence
(cf.~\cites{Higson-Roe-Yu:Coarse_Mayer-Vietoris}) that
the K-theory of \( \CstRoe(\R^d) \) coincides with the K-theory of a point
by a degree shift of \( -d \). This holds for any Delone
set \( \Lambda\subseteq\R^d \) as well by
\cref{prop:Roe_Cst-algebra_coarsely_invariant}. Therefore,
\K-theory of the real or complex Roe \Cst-algebras of a Delone set \(
\Lambda\subseteq\R^d \) is given by
\begin{equation}\label{eq:K-theory_Roe_Cst-algebras}
\begin{aligned}
\K_i(\CstRoe(\Lambda)_\C)&\simeq \begin{cases}
\Z{\hphantom{/2}} & \text{if } i-d\equiv 0{\hphantom{,5,6,7}}\mkern-10mu\mod 2; \\
0 & \text{if } i-d\equiv 1{\hphantom{,5,6,7}}\mkern-10mu\mod 2. 
\end{cases} \\
\K_i(\CstRoe(\Lambda)_\R)&\simeq \begin{cases}
\Z & \text{if } i-d\equiv 0,4{\hphantom{,6,7}}\mkern-10mu\mod 8; \\
\Z/2 & \text{if } i-d\equiv 1,2{\hphantom{,6,7}}\mkern-10mu\mod 8; \\
0 & \text{if } i-d\equiv 3,5,6,7\mkern-10mu\mod 8.
\end{cases}
\end{aligned}
\end{equation}
We note that the isomorphisms of abelian groups in 
\eqref{eq:K-theory_Roe_Cst-algebras} are actually given by
\( \K_*(\C) \)-module or \( \KO_*(\R) \)-module isomorphisms
\[
  \K_*(\CstRoe(\Lambda)_\C)\simeq\K_{*-d}(\C),\quad
  \KO_*(\CstRoe(\Lambda)_\R)\simeq\KO_{*-d}(\R),
\]
cf.~\cite{Ewert-Meyer:Coarse_geometry}*{Section 3 and Corollary 5}.

While using the Roe \Cst-algebra as the observable \Cst-algebra, there is
no need to pass to its matrix algebras, as opposed to the case of 
groupoids.  It was proven in 
\cite{Ewert-Meyer:Coarse_geometry}*{Corollary 1} that an
operator \( T \) on \( \mathcal{H}_\Lambda^{\oplus N} \) is locally compact
and controlled if and only if its matrix elements in \( \Bdd(\mathcal{H}_\Lambda) \)
are locally compact and controlled, when \( T \) is viewed as an \( N\times
N \)-matrix with entries in \( \Bdd(\mathcal{H}_\Lambda) \). Therefore
\[
\Mat_N\CstRoe(\Lambda,\mathcal{H}_\Lambda)\simeq\CstRoe(\Lambda,\mathcal{H}_\Lambda^{\oplus
N}) 
\] 
for all \( N \), 
where \( \mathcal{H}_\Lambda^{\oplus N} \) carries the direct sum of the
standard ample representation \(
\Co(\Lambda)\to\Bdd(\mathcal{H}_\Lambda) \).
We may further identify \(
\mathcal{H}_\Lambda^{\oplus}\simeq\ell^2(\Lambda,\mathcal{K})\otimes\R^N\simeq\ell^2(\Lambda,\mathcal{K}\otimes\R^N)\simeq\ell^2(\Lambda,\mathcal{K})
\) using any unitary isomorphism \( \mathcal{K}\simeq\mathcal{K}\otimes\R^N
\). Such an isomorphism preserves controlled or locally compact operators
using the characterisation in
\cref{ex:discrete_space_locally_compact_controlled}.   

\subsubsection{Position spectral triple over the Roe \Cst-algebra}
Fix a closed subset \( X\subseteq\R^d \) equipped with the subspace metric,
e.g.~a Delone set in \( \R^d \). A relation
between the Roe \Cst-algebra on \( X \) and the ``position'' operators has
been considered in \cite{Ewert-Meyer:Coarse_geometry}*{Section 2.2} as
follows. For \( t=(t_1,\dots,t_d)\in\R^d \), 
let \( \mathrm{e}_t \) be the bounded continuous function on \(
X\subseteq\R^d \) given by
\[ 
    \mathrm{e}_t(x)\defeq\mathrm{e}^{\mathrm{i}t\cdot
            x}=\mathrm{e}^{\mathrm{i}\sum_{j=1}^dt_jx_j},\qquad
            x=(x_1,\dots,x_d)\in X\subseteq\R^d.
\]
Then the map \( \R^d\to\Cb(X) \) given by \( t\mapsto \mathrm{e}_t \) is
continuous in the strict topology of \( \Cb(X) \).

Let \( \mathcal{H}_X \) be an ample \( X \)-module.
The representation \( \varrho\colon\Co(X)\to\Bdd(\mathcal{H}_X) \)
extends to a strictly continuous, \st-representation
\( \overline{\varrho}\colon\Cb(X)\to\Bdd(\mathcal{H}_X) \). Thus
the map
\[ 
    \sigma\colon\R^d\to\Cb(X)\to\Bdd(\mathcal{H}_X),\qquad t\mapsto
    \mathrm{e}_t\mapsto\varrho(\mathrm{e}_t)
\]
is continuous for the norm topology on \( \Bdd(\mathcal{H}_X) \). If \( X
\) does not contain discrete components, 
and \( \mathcal{H}_X \) is the standard \( X \)-module 
as in \cref{ex:standard_ample_module}, then the restriction of \( \sigma \) to
the \( j \)-th coordinate component of \( \R^d \) gives the flow 
\( \mathrm{e}^{\mathrm{i}t\pos_j} \) generated by \( \pos_j \).

By conjugation, \( \sigma\colon\R^d\to\Bdd(\mathcal{H}_X) \) induces a
group homomorphism \( \Ad\sigma\colon\R^d\to\Aut(\Bdd(\mathcal{H}_X)) \).
We say that an operator \( T\in\Bdd(\mathcal{H}_X) \) is continuous with
respect to \( \Ad\sigma \) if and only if the map \( t\mapsto\Ad\sigma_t(T) \) is
continuous for the norm topology on \( \Bdd(\mathcal{H}) \). The property
of being a norm limit of controlled operators can be described as a
continuity property for the action \( \Ad\sigma \): 

\begin{lemma}[\cite{Ewert-Meyer:Coarse_geometry}*{Theorem
4}]\label{lem:controlled_as_continuity}
An operator \( T\in\Bdd(\mathcal{H}_X) \) is a norm limit of controlled
operators if and only if it is continuous with respect to \( \Ad\sigma \). Thus \(
T\in\CstRoe(X,\mathcal{H}_X) \) if and only if it is locally compact and continuous
with respect to \( \Ad\sigma \). 
\end{lemma}
 
We take a slight different viewpoint, interpreting the previous lemma as
follows: the Roe \Cst-algebra is the largest \Cst-algebra, over which one
can have a position spectral triple:

\begin{theorem}\label{thm:largest_position_spectral_triple}
Let \( \Lambda\subseteq\R^d \) be a discrete subset. 
If 
\[  
    \xi\defeq \left(\mathcal{A}\otimes\Cl_{0,d},\quad
    \ell^2(\Lambda)\otimes\mathcal{K}\otimes\bigwedge\nolimits^*\R^{d},\quad
    \sum_{j=1}^d\pos_j\otimes\id_{\mathcal{K}}\otimes\gamma^j\right)
\]
is a position spectral triple
\textup{(}cf.~\cref{def:position_spectral_triple}\textup{)}, where \(
\mathcal{A} \) is a dense \st-subalgebra of a \Cst-algebra \( A \)  
represented on \( \ell^2(\Lambda)\otimes\mathcal{K} \) by
\( \varphi\colon A\to\Bdd(\ell^2(\Lambda)\otimes\mathcal{K}) \).
Then \( \varphi(\mathcal{A}) \) is contained in the
Roe \Cst-algebra \( \CstRoe(\Lambda,\mathcal{H}_\Lambda) \) 
defined by the standard \( \Lambda \)-module 
\( \mathcal{H}_\Lambda\defeq \ell^2(\Lambda)\otimes\mathcal{K} \)
as in \cref{ex:standard_ample_module}.
\end{theorem}

\begin{proof}
Identifying \( A \) and its dense subalgebra
\( \mathcal{A} \) with their images under 
\( \varphi \), we may assume that \( \varphi\colon
A\to\Bdd(\ell^2(\Lambda)\otimes\mathcal{K}) \) 
is injective. Then we must show that every 
\( a\in\mathcal{A} \) is locally compact and
can be approximated by controlled operators. 
Being locally compact amounts to saying that \(
\braketvert{x}{\varphi(a)}{y}\in\Cpt(\mathcal{K}) \) for all \(
x,y\in\Lambda \), which always holds by 
\cref{lem:automatic_local_compactness} (providing that \( \xi \) \emph{is}
a position spectral triple.)

Now we show that every \( a\in\mathcal{A} \) is a norm limit of controlled
operators. For convenience, below we write \( \hat{\pos}_j \) for the
operator \(
\pos_j\otimes\id_\mathcal{K}\) on \( \ell^2(\Lambda)\otimes\mathcal{K} \).  

Since \( [a\otimes\rho^i,\hat{\pos}_j\otimes\gamma^j] \) is bounded 
for all \( a\in\mathcal{A} \)  and
\( i,j\in\{1,\dots,d\} \), we have \( [a,\hat{\pos}_j] \) 
is a bounded operator
for all \( j\in\{1,\dots,d\} \), i.e.~there exists \( C>0 \) such that \(
\norm*{[a,\hat{\pos}_j]}\leq C \).

By \cref{lem:controlled_as_continuity}, it suffices to show that
for all \( a\in\mathcal{A} \) and \( j=1,\dots,d \), the map
\[ 
    \R\to\Bdd(\ell^2(\Lambda,\mathcal{K})),\qquad t\mapsto
    \mathrm{e}^{\mathrm{i}t\hat{\pos}_j}a
    \mathrm{e}^{-\mathrm{i}t\hat{\pos}_j}
\]
is continuous in the norm topology on \(
\Bdd(\ell^2(\Lambda,\mathcal{K})) \). It follows from
\[ 
[\mathrm{e}^{it\hat{\pos}_j},a]
=\int_0^1\frac{\mathrm{d}}{\mathrm{d}s}\left(\mathrm{e}^{\mathrm{i}st\hat{\pos}_j}a\mathrm{e}^{\mathrm{i}(1-s)t\hat{\pos}_j}\right)\mathrm{d}s
=\mathrm{i}t\int_0^1
\mathrm{e}^{\mathrm{i}st\hat{\pos}_j}[\hat{\pos}_j,a]\mathrm{e}^{\mathrm{i}(1-s)t\hat{\pos}_j}\mathrm{d}s
\]
that
\[ 
    \norm*{[\mathrm{e}^{\mathrm{i}t\hat{\pos}_j},a]}\leq
    t\int_0^1\norm*{[\hat{\pos}_j,a]}\mathrm{d}s\leq Ct.
\]
Therefore,
\begin{align*}
\norm*{\mathrm{e}^{\mathrm{i}s\hat{\pos}_j}a\mathrm{e}^{-\mathrm{i}s\hat{\pos}_j}
-\mathrm{e}^{\mathrm{i}t\hat{\pos}_j}a\mathrm{e}^{-\mathrm{i}t\hat{\pos}_j}}
&=\norm*{\mathrm{e}^{\mathrm{i}t\hat{\pos}_j}\left(\mathrm{e}^{\mathrm{i}(s-t)\hat{\pos}_j}a\mathrm{e}^{-\mathrm{i}(s-t)\hat{\pos}_j}-a\right)\mathrm{e}^{-\mathrm{i}t\hat{\pos}_j}}
\\
&=\norm*{\mathrm{e}^{\mathrm{i}t\hat{\pos}_j}[\mathrm{e}^{\mathrm{i}(s-t)\hat{\pos}_j},a]\mathrm{e}^{-\mathrm{i}(s-t)\hat{\pos}_j}\mathrm{e}^{-\mathrm{i}t\hat{\pos}_j}}
\\
&\leq \norm*{[\mathrm{e}^{\mathrm{i}(s-t)\hat{\pos}_j},a]} \\
&\leq C(s-t).
\end{align*}
So the map \( t\mapsto \Ad{\sigma_t}(a) \) is continuous for all \( a\in
\mathcal{A} \), and thus \(
a\in\mathcal{A} \) is a norm limit of controlled operators.
\end{proof}

Thus if the \st-algebra \( \mathcal{A} \) in the position spectral triple
is dense in \( \CstRoe(\Lambda) \), then it defines a spectral triple over
\( \CstRoe(\Lambda) \). This amounts to choosing a suitable dense
\st-subalgebra inside \( \CstRoe(\Lambda) \). The spectral triple below 
has been implicitly used by Ewert and Meyer in the proof of
\cite{Ewert-Meyer:Coarse_geometry}*{Theorem 7}.

\begin{deflem}
Let \( \Lambda\subseteq\R^d \) be a discrete countable set. 
Let \( \CRoe(\Lambda) \) be the collection of operators
\( T\in\Bdd(\ell^2(\Lambda,\mathcal{K})) \) satisfying:
\begin{itemize}
\item \( T \) is controlled and locally compact\textup{;}
\item \( \sup_{y\in \Lambda}\sum_{x\in \Lambda}\norm*{T_{x,y}}<+\infty \) and
\( \sup_{x\in \Lambda}\sum_{y\in \Lambda}\norm*{T_{x,y}}<+\infty \).
\end{itemize}
Then
\begin{equation}\label{eq:Roe_position_spectral_triple}
  \xi^\Roe_\Lambda\defeq\left(\CRoe(\Lambda)\otimes\Cl_{0,d},\quad\ell^2(
  \Lambda)\otimes\mathcal{K}\otimes\bigwedge\nolimits^*\R^{d},\quad\sum_{j=1}^d\pos_j\otimes\id_\mathcal{K}\otimes\gamma^j\right)
\end{equation}
is a position spectral triple over \(
\CstRoe(\Lambda) \). 
\end{deflem}

\begin{proof}
Let \( T\in\CRoe(\Lambda) \), \( x,y\in \Lambda \). Then
\[
\braketvert{x}{[T,\pos_j\otimes\id_\mathcal{K}]}{y}
=\braketvert{x}{T(\pos_j\otimes\id_\mathcal{K})}{y}-\braketvert{x}{(\pos_j\otimes\id_\mathcal{K})T}{y}
=(y_j-x_j)T_{x,y}.
\]
Since \( T \) is controlled, there exists \( R>0 \) such that \(
\braketvert{x}{T}{y}=0 \) whenever \( \mathrm{d}(x,y)>R \). Thus     
\[ 
    \norm*{[T,\pos_j\otimes\id_\mathcal{K}]}\leq\sup_{y\in \Lambda}
    \sum_{x\in \Lambda}R\cdot\norm*{T_{x,y}}<+\infty.
\]
So \( T \) preserves the 
domain of \( \pos_j\otimes\id_\mathcal{K} \)
and boundedly commutes with it for
every \( j=1,\dots,d \).

Therefore, for any \( S\in\Cl_{0,d} \), \( T\otimes S \) preserves the domain of
\( \pos\defeq \sum_{j=1}^n\pos_j\otimes\id_\mathcal{K}\otimes\gamma^j \),
and boundedly commutes with \( \pos \). 
The local compactness of \( T \) implies that the operator
\[
T\left(1+\left(\sum_{j=1}^n\pos_j\otimes\id_\mathcal{K}\otimes\gamma^j\right)^2\right)^{-1}=T\left(1+\sum_{j=1}^n\pos_j^2\right)^{-1}\otimes\id_\mathcal{K}\otimes\id_{\bigwedge\nolimits^*\R^{d}}
\]
is compact. Therefore, \( \xi^\Roe_\Lambda \) is a position spectral triple.
\end{proof}

\subsubsection{Position spectral triple generates \( \KO_*(\R) \)}
Towards the end of this section, we shall prove that the position spectral
triple \( \xi_\Lambda^\Roe \) induces isomorphisms
\[ 
    \KK(\Cl_{n,0},\CstRoe(\Lambda))\xrightarrow{\sim}\KK(\Cl_{n,d},\R)
\]
for all \( n \). The proof is based on these observations.
\begin{enumerate}
\item If \( \Lambda=\Z^d \), then 
the KK-class of the position spectral triple
\[ 
[\xi^\Roe_{\Z^d}]\in\KK(\CstRoe(\Z^d)_\R\otimes\Cl_{0,d},\R)
\]
pulls back to the fundamental class of \(
\KK(\Cst(\Z^d)_\R\otimes\Cl_{0,d},\R) \) as in
\cref{thm:fundamental_class_Zd}. 
Therefore, for each \( n \), it induces an isomorphism
\[
  \KK(\Cl_{n,0},\CstRoe(\Z^d))\xrightarrow{\sim}\KK(\Cl_{n,d},\R).
\]
\item Any other Delone set \(
\Lambda\subseteq\R^d \) is coarsely equivalent to
\( \Z^d \). Therefore, \( \xi^\Roe_{\Lambda} \) induces the same map as \(
\xi^\Roe_{\Z^d} \) up to the isomorphism given by
\[
    \KK(\Cl_{n,0},\CstRoe(\Z^d))\simeq\KK(\Cl_{n,0},\CstRoe(\Lambda)).
\]
\end{enumerate}

\begin{lemma}\label{lem:Roe_Dirac_element_Z^d}
The position spectral triple for \( \Lambda=\Z^d \)
\begin{equation}\label{eq:spectral_triple_Roe_Z^d}
\xi^\Roe_{\Z^d}\defeq
\left(\CRoe(\Z^d)\otimes\Cl_{0,d},\quad\ell^2(\Z^d)\otimes\mathcal{K}\otimes\bigwedge\nolimits^*\R^d,\quad\sum_{j=1}^d\pos_j\otimes\id_\mathcal{K}\otimes\gamma^j\right),
\end{equation}
induces, for each \( n \), an isomorphism  
\[ 
  \KK(\Cl_{n,0},\CstRoe(\Z^d))\xrightarrow{\sim}\KK(\Cl_{n,d},\R).
\]
\end{lemma}

\begin{proof}
The proof comes from \cite{Ewert-Meyer:Coarse_geometry}*{Theorem 7,
Corollary 5}, which we sketch here.

First let \( n=d \) and \( \iota\colon \Cst(\Z^d)\hookrightarrow
\CstRoe(\Z^d) \) be the inclusion map induced by the obvious isometry
\( \ell^2(\Z^d)\hookrightarrow \ell^2(\Z^d)\otimes\mathcal{K} \). 
Then
\[ 
  \xi_{\Z^d,1}=\iota^*\xi_{\Z^d}^\Roe,
\]
that is, \( \xi_{\Z^d,1} \) is the pullback of the position spectral triple
over \( \CstRoe(\Z^d)_\R \). 
Thus the following diagram commutes by functorality of the Kasparov
product (cf.~\cref{lem:Kasparov_product_functoriality}):
\[ \begin{tikzcd}
\KK(\Cl_{d,0},\Cst(\Z^d)_\R) \arrow[rr,
"{\times_{\Cst(\Z^d)_\R}[\xi_{\Z^d,1}]}"] \arrow[rd, "{\times_A[\iota]}"']
&& \KK(\Cl_{d,d},\R)\simeq\Z \\
& \KK(\Cl_{d,0},\CstRoe(\Z^d)_\R) \arrow[ru, "{\times_{\CstRoe(\Z^d)_\R}[\xi^\Roe_{\Z^d}]}"'] 
\end{tikzcd} \]
and the horizontal arrow is a surjective group homomorphism by
\cref{thm:fundamental_class_Zd}. Thus the
bottom right arrow is a surjective group homomorphism \( \Z\to\Z \), hence
an isomorphism.

The isomorphisms of abelian groups \(
\KO_n(\CstRoe(\Z^d)_\R)\simeq\KO_{n-d}(\R) \) give a  
give an isomorphism of \(
\KO_*(\R) \)-modules \( \KO_*(\CstRoe(\Z^d)_\R)\simeq\KO_{*-d}(\R) \)
(cf.~the discussion after 
\cite{Ewert-Meyer:Coarse_geometry}*{Proposition 8}).
Functorality of the Kasparov product implies that \( \xi_{\Z^d}^\Roe \)
induces a \( \KO_*(\R) \)-module homomorphism \(
\KO_*(\CstRoe(\Z^d)_\R)\to\KO_{*-d}(\R) \), both sides being isomorphic to
free \( \KO_*(\R) \)-modules with a generator in degree \( -d \). The
commutative diagram above shows that \( [\xi_{\Z^d}^\Roe] \) maps
such a generator to a generator, hence a \( \KO_*(\R) \)-module
isomorphism. In particular, it follows that the induced maps 
\[ 
  \xi_{\Z^d}^\Roe\colon\underbrace{\KK(\Cl_{n,0},\CstRoe(\Z^d)_\R)}_{\simeq\KO_n(\CstRoe(\Z^d)_\R)}\to\underbrace{\KK(\Cl_{n,d},\R)}_{\simeq\KO_{n-d}(\R)}
\]
are isomorphisms of abelian groups for all \( n \).
\end{proof}

\begin{lemma}\label{lem:adjoints_induce_KK-equivalences}
Let \( L_1,L_2\subseteq\R^d \) be Delone sets such that \( L_1\cap
L_2=\emptyset \). Let \( L\defeq L_1\sqcup L_2 \). 
For \( i\in\{1,2\} \), we write\textup{:}
\begin{itemize}
\item \( \mathcal{H}_i\defeq \ell^2(L_i,\mathcal{K}) \) for the standard
ample \( L_i \)-module, and \(
\mathcal{H}\defeq \ell^2(L_1\sqcup L_2,\mathcal{K}) \) for the standard
ample \( L \)-module\textup{;}
\item \( \xi^\Roe_{L_i} \) and \( \xi^\Roe_L \)  for the position spectral
triples over \( \CstRoe(L_i) \) and \( \CstRoe(L) \) as in
\eqref{eq:Roe_position_spectral_triple}\textup{;}
\item \( \iota_i\colon L_i\hookrightarrow L \) for the isometric
embeddings\textup{;}
\item \( V_i\colon \mathcal{H}_i\hookrightarrow \mathcal{H}_0 \) for the
isometries induced by \( \iota_i \), that is, \(
V_i\ket{x}\defeq\ket{\iota_i(x)} \) for all \( x\in L_i \).
\end{itemize}
Then the following hold\textup{:}
\begin{enumerate}
\item \( V_i's \) are covering isometries for \( \iota_i \). Hence
\( \Ad_{V_i}\colon T\mapsto V_iTV_i^* \) maps \( \CstRoe(L_i) \) into
\( \CstRoe(L) \), and induces an isomorphism
\[ 
    \iota_{i*}\colon\KK(\Cl_{n,0},\CstRoe(L_i))\to\KK(\Cl_{n,0},\CstRoe(L)).
\]
\item The following diagram commutes (all arrows are given by taking the
Kasparov product)\textup{:}
\begin{equation}\label{eq:commutative_diagram_Z^d_Lambda}
\begin{tikzcd}
\KK(\Cl_{n,0},\CstRoe(L_1)) \arrow[r,"\iota_{1*}","\sim"']
\arrow[dr,"\xi^\Roe_{L_1}"'] &
\KK(\Cl_{n,0},\CstRoe(L))
\arrow[d,"\xi^\Roe_L"] &
\KK(\Cl_{n,0},\CstRoe(L_2)) \arrow[l,"\iota_{2*}"',"\sim"]
\arrow[dl,"\xi^\Roe_{L_2}"] \\
& \KK(\Cl_{n,d},\R) &
\end{tikzcd} 
\end{equation}
\end{enumerate}
\end{lemma}

\begin{proof}
(1) follows from \cref{prop:Roe_Cst-algebra_coarsely_invariant} as both \(
L_1 \) and \( L_2 \)  are coarsely dense in \( L \). 
Since \( L_1\cap L_2=\emptyset \) and \( L=L_1\sqcup L_2 \), 
it follows that
\[ 
  \mathcal{H}_1\oplus\mathcal{H}_2\xrightarrow{\sim}\mathcal{H},\quad
  (\phi_1,\phi_2)\mapsto V_1\phi_1+V_2\phi_2.
\]
is a unitary isomorphism of Hilbert spaces.

Let \( \pos^i_j \) be the \( j \)-th position operator on \( \ell^2(L_i)
\), that is,
\[ 
  \pos_j^if(x)\defeq x_jf(x),\quad \phi\in\Cc(L_i),\;x\in L_i.
\]
Also write \( \pos_j \) for the \( j \)-th position operator on \(
\ell^2(L) \). Then \( \pos_j \) restricts to the \( j \)-th position
operator on \( \ell^2(L_i) \). That is, we have  
\[ 
    \pos_j=V_{1}\pos_j^{1}V_{1}^*\oplus
    V_{2}\pos_j^{2}V_{2}^*=\Ad_{V_1}(\pos_j^1)\oplus\Ad_{V_2}(\pos_j^2),
\]
and hence
\[ 
    \sum_{j=1}^d\pos_j\otimes\id_\mathcal{K}\otimes\gamma^j
    =\Ad_{V_1}\left(\sum_{j=1}^d\pos_j^1\otimes\id_\mathcal{K}\otimes\gamma^j\right)
    \oplus
    \Ad_{V_2}\left(\sum_{j=1}^d\pos_j^2\otimes\id_\mathcal{K}\otimes\gamma^j\right).
\]
Therefore, for \( i\in\{1,2\} \), 
the \st-homomorphism \( \Ad_{V_{j}} \)  pulls back the
class \( [\xi_L^\Roe]\in\KK(\CstRoe(L)\otimes\Cl_{0,d},\R) \) to the
class \( [\xi_{L_i}^\Roe]\in\KK(\CstRoe(L_i)\otimes\Cl_{0,d},\R) \).  
So both the left and the right triangle commutes 
by \cref{lem:Kasparov_product_functoriality}.
\end{proof}

\begin{theorem}\label{thm:position_spectral_triple_Roe_isomorphism}
For any Delone set \( \Lambda \), 
the position spectral triple \( \xi^\Roe_\Lambda \) 
induces, for each \( n \), an isomorphism
\[
    \KK(\Cl_{n,0},\CstRoe(\Lambda))\xrightarrow{\sim}\KK(\Cl_{n,d},\R).
\]
\end{theorem}

\begin{proof}
The first step is to replace \( \Lambda \) by its translated copy 
that is disjoint from \( \Z^d \). To be precise, we note that the set
\[ 
  B\defeq \set{x-y\mid x\in\Lambda,y\in\Z^d}
\]
is countable, so the set \( \R^d\setminus B \) is non-empty. Choose any \(
v\in\R^d\setminus B \), then the Delone
set \( \Lambda'\defeq \Lambda-v \) is disjoint from \( \Z^d \). The
translation map
\[ 
  f\colon \Lambda\to\Lambda',\quad x\mapsto x-v
\]
induces a unitary isomorphism of Hilbert spaces 
\( V\colon \ell^2(\Lambda)\xrightarrow{\sim}\ell^2(\Lambda')
\) that covers \( f \). Hence \( V \) gives a \st-isomorphism 
\( \Ad_V\colon
\CstRoe(\Lambda)\xrightarrow{\sim}\CstRoe(\Lambda') \), which pulls back
the position spectral triple \( \xi_{\Lambda'}^\Roe \) to the position
spectral triple \( \xi_{\Lambda}^\Roe \). 

As a consequence, we may assume without loss of generality that \(
\Z^d\cap\Lambda=\emptyset \). Set \( L_1\defeq \Z^d \) and \( L_2\defeq
\Lambda \).     
Then the position spectral triple \( \xi^\Roe_{L_1}=\xi^\Roe_{\Z^d} \) 
induces an isomorphism 
\[ 
    \KK(\Cl_{n,0},\CstRoe(\Z^d))\xrightarrow{\sim}\KK(\Cl_{n,d},\R)
\]
by \cref{lem:Roe_Dirac_element_Z^d}. The isometry \( V_{1} \) induces an
isomorphism between the corresponding Kasparov groups by
\cref{lem:adjoints_induce_KK-equivalences}. Thus the commutativity of the
left triangle in \eqref{eq:commutative_diagram_Z^d_Lambda} implies that the
map
\[ 
    \KK(\Cl_{n,0},\CstRoe(\Z^d\sqcup\Lambda))\to\KK(\Cl_{n,d},\R)
\]
given by the position spectral triple \( \xi^\Roe_{\Z^d\sqcup\Lambda} \),
must be an isomorphism. Thus the position spectral triple \(
\xi_{L_2}^\Roe=\xi_\Lambda^\Roe \) must be an isomorphism as well, as it is 
a composition of isomorphisms
\[ 
    \KK(\Cl_{n,0},\CstRoe(\Lambda))
    \underset{\sim}{\xrightarrow{\iota_{1*}}}\KK(\Cl_{n,0},\CstRoe(\Z^d\sqcup\Lambda))
    \underset{\sim}{\xrightarrow{\xi_{L}^\Roe}}\KK(\Cl_{n,d},\R).
    \qedhere
\]
\end{proof}

\subsubsection{Remarks on uniform Roe \Cst-algebras}
\label{sec:uniform_Roe_Cst-algebra}

Let \( \Lambda \) be a discrete metric space. An operator \(
T\in\Bdd(\ell^2(\Lambda)) \) is called \emph{controlled}, if there exists
\( R>0 \) such that \( T_{x,y}\defeq \braketvert{x}{T}{y}=0 \) whenever \(
x,y\in\Lambda \) satisfy \(
\mathrm{d}(x,y)>R \).  The \emph{uniform Roe \Cst-algebra} of \( \Lambda
\), denoted by \( \CstuRoe(\Lambda) \), is the \Cst-algebra generated by
all controlled operators on \( \ell^2(\Lambda) \). 

Starting from Kubota \cite{Kubota:Controlled_topological_phases}, uniform
Roe \Cst-algebras have also been employed as the observable \Cst-algebras of
aperiodic topological insulators. As opposed to Roe \Cst-algebras, uniform
Roe \Cst-algebras have much more involved K-theory. For example, it is
known that \( \K_0(\CstuRoe(\Z)) \) is uncountable for \( d\geq 1 \)
(cf.~\cite{Spakula:K-theory_uniform_Roe_algebras}*{Example II.3.4}).  

We briefly remark here that the uniform Roe \Cst-algebra is also quite
``universal'' as it contains all ``tight-binding'' observable \Cst-algebras,
in the following sense:

\begin{proposition}
Assume that there is a real spectral triple of the form
\begin{equation}
\zeta_\Lambda\defeq\left(\mathcal{A}\otimes\Cl_{0,d},\quad\ell^2(\Lambda)\otimes\bigwedge\nolimits^*\R^{d},\quad\pos\defeq
\sum_{j=1}^d\pos_j\otimes\gamma^j\right),
\end{equation}
where:
\begin{itemize}
\item \( A \) is a real \Cst-algebra, 
which carries a \st-representation \(
\varphi\colon A\to\Bdd(\ell^2(\Lambda)) \); 
\( \mathcal{A}\subseteq A \) is a dense \st-subalgebra; 
\item \( \pos_j \) is the \( j \)-th position operator on
\(
\ell^2(\Lambda) \), given by
\[ 
    (\pos_j\phi)(x)\defeq x_j\phi(x),\qquad \phi\in\Cc(\Lambda),\;x=(x_1,\dots,x_d)\in\Lambda\subseteq\R^d;
\]
\item \( \gamma_1,\dots,\gamma_d \) are the generators of \( \Cl_{d,0} \),
represented on \( \bigwedge\nolimits^*\R^{d} \) via the standard representation. 
\end{itemize}
Then \( \varphi(A) \) is contained in \( \CstuRoe(\Lambda) \).  
\end{proposition}

We note that the auxiliary Hilbert space \( \mathcal{K} \) is now absent,
as opposed to the definition of a position spectral triple
(cf.~\cref{def:position_spectral_triple}). Thus \( A \) is merely
represented on the ``tight-binding'' Hilbert space \( \ell^2(\Lambda) \)
instead of \( \ell^2(\Lambda)\otimes\mathcal{K} \). 
Replacing \( A \) by \( A\otimes\Mat_N(\R) \) replaces \( \ell^2(\Lambda)
\) by \( \ell^2(\Lambda,\R^N)=\ell^2(\Lambda)\otimes\R^N \).

\begin{proof}
It follows from \cref{lem:controlled_as_continuity} and the proof of
\cref{thm:largest_position_spectral_triple} that given the spectral triple
\( \zeta_\Lambda \), then every \( a\in\mathcal{A} \) is represented as a
controlled operator on \( \ell^2(\Lambda) \). Thus \(
\varphi(\mathcal{A})\subseteq \CstuRoe(\Lambda) \) and hence \(
\varphi(A)\subseteq \CstuRoe(\Lambda) \).  
\end{proof}

Thus tight-binding models of topological insulators, e.g.~the
groupoid model \( \Cst(\mathcal{G}_\Lambda) \), 
is already
contained in the uniform Roe \Cst-algebra \( \CstuRoe(\Lambda) \), and
we have inclusions of \Cst-algebras and isometric embeddings of Hilbert
spaces which they act on:
\[ \begin{tikzcd}[column sep=small, row sep=0.1em]
\Cst(\mathcal{G}_\Lambda)\otimes\Mat_N(\R) \arrow[r, hook]
& \CstuRoe(\Lambda)\otimes\Mat_N(\R) \arrow[r, hook]
& \CstuRoe(\Lambda)\otimes\Cpt(\mathcal{K}) \arrow[r, hook]
& \CstRoe(\Lambda) \\
\curvearrowright & \curvearrowright & \curvearrowright & \curvearrowright
\\
\ell^2(\Lambda)\otimes\R^N \arrow[r, hook]
& \ell^2(\Lambda)\otimes\R^N \arrow[r, hook]
& \ell^2(\Lambda)\otimes\mathcal{K} \arrow[r, hook]
& \ell^2(\Lambda)\otimes\mathcal{K}.
\end{tikzcd} \]

As argued in \cite{Ewert-Meyer:Coarse_geometry}*{Section 1}, the occurrence
of uniform Roe \Cst-algebras might be considered as an artefact of the
tight-binding approximation.

\section{Robustness of topological phases}
Now we use the results from previous sections to compare the groupoid model
\( \Cst(\mathcal{G}_\Lambda) \) with the coarse-geometric model \(
\CstRoe(\Lambda) \) of topological phases on an aperiodic lattice \(
\Lambda \). Topological phases described by the K-theory of \(
\CstRoe(\Lambda) \) are called \emph{strong} in
\cite{Ewert-Meyer:Coarse_geometry}. We follow this terminology. 
Such phases are
stable under perturbations which perserve the conjugate-linear and/or
chiral symmtries of
the system, are locally compact and controlled, and do
not close the gap of the Hamiltonian. This is because such perturbations
lift to homotopies in the Roe \Cst-algebra, and hence preserve K-theory.

We will explain which phases, described by the K-theory of 
\( \Cst(\mathcal{G}_\Lambda) \), have the same robustness. This is done by 
mapping \( \Cst(\mathcal{G}_\Lambda) \) into a Roe \Cst-algebra 
\( \CstRoe(\omega) \) using the localised regular representations \(
\pi_\omega \) depending on a choice of \( \omega\in\Omega \),
whereas all of these
\( \omega \) yield isomorphic Roe \Cst-algebras as they are Delone sets in \( \R^d \). 
We also explain
that ``stacked'' topological phases, coming from lower-dimensional Delone
sets, are always weak. Both results can be viewed as generalisations of
\cite{Ewert-Meyer:Coarse_geometry}*{Section 4}, in which the Delone set \(
\Lambda \) is the periodic square lattice \( \Z^d \).  

\subsection{Position spectral triples detect strong topological phases}
The results of \cref{thm:position_spectral_triple_groupoid} and
\cref{thm:position_spectral_triple_Roe_isomorphism} allows us to compare
the groupoid model and the coarse-geometric model on the level of both
\Cst-algebras and K-theory (index pairing). The main theorem is the
following:

\begin{theorem}[Position spectral triples detect strong topological phases]
\label{thm:position_detects_strong_phases}
For every \( \omega\in\Omega_0 \), The following diagram commutes:
\begin{equation}\label{eq:position_detects_strong_phases}
\begin{tikzcd}[column sep=large, row sep=large]
\KK(\Cl_{n,0},\Cst(\mathcal{G}_\Lambda)) \arrow[r, "\bulkcycle"] 
\arrow[d,"\pi^N_\omega"'] \arrow[dr, "\xi^\Gpd_{\omega,N}"]
& \KK(\Cl_{n,d},\Cont(\Omega_0)) \arrow[d,
"\ev_\omega"] \\
\KK(\Cl_{n,0},\CstRoe(\omega)) \arrow[r, "\xi^\Roe_{\omega}", "\sim"'] &
\KK(\Cl_{n,d},\R).
\end{tikzcd} 
\end{equation}
where:
\begin{itemize}
\item
\( \pi_\omega^N\defeq \pi_\omega\otimes
e_N\colon\Cst(\mathcal{G}_\Lambda)\otimes\Mat_N(\R)\to\Bdd(\ell^2(\omega)\otimes\mathcal{K}) \)
is the entrywise extension of the localised regular representation at \(
\omega\in\Omega_0 \) as in \eqref{eq:pi_omega^N}\textup{;}
\item \( \bulkcycle \) is the bulk cycle as in 
\eqref{eq:bulk_cycle}\textup{;}
\item \( \xi^\Gpd_{\omega,N} \) is the spectral triple over \(
\Cst(\mathcal{G}_\Lambda)\otimes\Mat_N(\R) \) defined in
\eqref{eq:groupoid_position_spectral_triple}\textup{;} 
\( \xi^\Roe_\omega \) is the
spectral triple over \( \CstRoe(\omega) \) defined in
\eqref{eq:Roe_position_spectral_triple}\textup{;}
\item all arrows are given by taking Kasparov products.
\end{itemize}
\end{theorem}

\begin{proof} 
We claim that \(
\pi^N_\omega\colon\Cst(\mathcal{G}_\Lambda)\otimes\Mat_N(\R)
\to\Bdd(\ell^2(\omega)\otimes\mathcal{K})
\) maps into \( \CstRoe(\omega) \) defined by the standard ample \( \omega
\)-module \( \ell^2(\omega,\mathcal{K}) \).  
Whenever this holds, then it follows that the KK-class represented
by the spectral triple \eqref{eq:groupoid_position_spectral_triple}
\[ 
\xi_{\omega,N}^\Gpd\defeq\left(\Mat_N(\Cc(\mathcal{G}_\Lambda))\otimes\Cl_{0,d},\quad\ell^2(\omega)\otimes\mathcal{K}\otimes\bigwedge\nolimits^*\R^{d},\quad\sum_{j=1}^d\pos_j\otimes\id_\mathcal{K}\otimes\gamma^j\right)
\]
equals the pullback along \( \pi^N_\omega \) 
of the KK-class represented by
\eqref{eq:Roe_position_spectral_triple}:
\[ 
\xi^\Roe_\omega\defeq\left(\CRoe(\omega)\otimes\Cl_{0,d},
\quad\ell^2(\omega)\otimes\mathcal{K}\otimes\bigwedge\nolimits^*\R^{d},
\quad\sum_{j=1}^d\pos_j\otimes\id_\mathcal{K}\otimes\gamma^j\right).
\]
We have shown in \cref{thm:position_spectral_triple_groupoid} that \(
\xi^\Gpd_{\omega,N} \) is a position spectral triple. This
implies that \( \pi_\omega(\Cc(\mathcal{G}_\Lambda)\otimes\Mat_N(\R))
\subseteq\CstRoe(\Lambda) \) by \cref{thm:largest_position_spectral_triple}. Since
\( \Cc(\mathcal{G}_\Lambda)\otimes\Mat_N(\R) \) is dense in \(
\Cst(\mathcal{G}_\Lambda)\otimes\Mat_N(\R) \), the image
of \( \pi_\omega^N \) is contained in \( \CstRoe(\omega) \). Then
it follows from
the functoriality
of the Kasparov product in \cref{lem:Kasparov_product_functoriality} 
that the bottom left triangle commutes.

The commutativity of the top right triangle follows
from the construction of \( \xi^\Gpd_\omega \) in
\cref{thm:position_spectral_triple_groupoid}. Therefore, we conclude that
the entire diagram commutes.
\end{proof}

\begin{example}
Consider the periodic square lattice \( \Lambda=\Z^d \). Then \(
\Omega_\Lambda\simeq\T^d \) and \( \Omega_0=\{\omega\} \) is a singleton,
in which \( \omega \) can be chosen to be any point in \( \Omega_\Lambda \)
(cf.~\cite{Bourne-Mesland:Topological_phases}*{Example 2.7}). 
There is a unique evaluation point \( \ev_\omega\colon\{\omega\}\to\C \), 
which gives an isomorphism 
\( \KO^0(\operatorname{pt})=\KO_0(\R)\xrightarrow{\sim}\Z \). 
The spectral
triple \( \xi^\Gpd_{\omega,N} \) is nothing but the position spectral
triple \( \xi_{\Z^d,N} \) over the real group \Cst-algebra \( \Cst(\Z^d)_\R
\) in \eqref{eq:spectral_triple_group_Z^d}. And the diagram
\eqref{eq:position_detects_strong_phases} recovers 
the commutative diagram in 
\cite{Ewert-Meyer:Coarse_geometry}*{Theorem 7}.
\end{example}

\subsection{Stacked topological phases are weak}
It was explained in \cite{Ewert-Meyer:Coarse_geometry} why certain
topological phases described by \( \K_*(\Cst(\Z^d)) \) are ``weak''.
Let \( \varphi\colon\Z^{d}\to\Z^{d+1} \) is an injective group homomorphism. 
It induces a map 
\( \varphi_*\colon\Cst(\Z^{d})\to\Cst(\Z^{d+1}) \) and hence maps
\[ 
    \varphi_*\colon\K_*(\Cst(\Z^{d}))\to\K_*(\Cst(\Z^{d+1}))
\]
in K-theory. Topological phases that belong to the image of \( \varphi_* \)
can be thought of as ``stacked'' from the lower-dimensional lattice \( \Z^d
\). It was shown in \cite{Ewert-Meyer:Coarse_geometry}*{Proposition 10} 
that such phases are killed by the map 
\( \K_*(\Cst(\Z^{d+1}))\to\K_*(\CstRoe(\Z^{d+1})) \) 
induced by the inclusion \( \Cst(\Z^{d+1})\hookrightarrow\CstRoe(\Z^{d+1})
\). The proof is based on the fact that
\( \varphi \) also induces a map between the Roe \Cst-algebras of \( \Z^d
\) and \( \Z^{d+1} \), which factors through a flasque space if \( \varphi
\) is injective. Then this map induces zero maps in K-theory.

We shall explain in this section how this observation can be generalised to
aperiodic lattices and higher-dimensional cases.
That is, we consider
Delone sets of the form \( \Lambda\times L \),
where \( L \) is another Delone set. If \( L \) has dimension \( 1 \), then
we may think of \( \Lambda\times L \) as
``stacking'' \( \Lambda \) along the direction of \( L \)
(cf.~\cref{fig:stacked_topological_phases}).
This also gives a \st-homomorphism groupoid
 \st-homomorphism \( \varphi^\Gpd\colon\Cst(\mathcal{G}_\Lambda)
\to\Cst(\mathcal{G}_{\Lambda\times L}) \), which induces maps
\[ 
    \varphi^\Gpd_*\colon\KK(\Cl_{n,0},\Cst(\mathcal{G}_\Lambda))\to\KK(\Cl_{n,0},\Cst(\mathcal{G}_{\Lambda\times
    L}))
\]
for all \( n \). 
Such maps can be interpreted as ``stacking'' topological phases living on
\( \Lambda \) along the direction \( L \). We will show that such
topological phases are always weak, in the sense that they vanish in the
K-theory of Roe \Cst-algebras.

\begin{figure}[htbp]
\centering
\tdplotsetmaincoords{60}{120}
\begin{tikzpicture}[tdplot_main_coords]
\def\NumPlanes{6}
\def\Spacing{0.25}
\def\PlaneSize{2}
\draw[->, thick] (0,0,0) -- (3,0,0) node[anchor=north east]{$x$};
\draw[->, thick] (0,0,0) -- (0,3,0) node[anchor=north west]{$y$};
\draw[->, thick] (0,0,0) -- (0,0,3) node[anchor=south]{$z$};
\foreach \i in {0,...,\numexpr\NumPlanes-1\relax} {
    \begin{scope}
        \pgfmathsetmacro{\z}{\i * \Spacing}
        \draw[fill=blue!10, draw=blue!70] 
            (0,0,\z) -- (\PlaneSize,0,\z) -- (\PlaneSize,\PlaneSize,\z) -- (0,\PlaneSize,\z) -- cycle;
    \end{scope}
}
\end{tikzpicture}
\caption{``Stacked'' topological phases along the \( z \)-direction}
\label{fig:stacked_topological_phases}
\end{figure}
We fix the following notation. Let \( \Lambda\subseteq\R^p \) be an \(
(r_\Lambda,R_\Lambda)
\)-Delone set and \( L\subseteq\R^q \) be an \( (r_L,R_L) \)-Delone
set. We write 
\( \Omega_{\Lambda\times L} \), \( \Omega_{\Lambda} \) and \( \Omega_L
\) for the closure of the orbit of \( \Lambda\times L \), \( \Lambda \)
and \( L \). To distinguish, we write
\[
\Omega^{\Lambda\times L}_0\defeq\set{\mu\in\Omega_{\Lambda\times L}\mid0\in
\mu},\quad
\Omega^\Lambda_0\defeq\set{\omega\in\Omega_{\Lambda}\mid0\in\omega}
,\quad
\Omega^L_0\defeq\set*{\ell\in \Omega_L\;\middle|\;0\in\ell}
\]
for corresponding abstract transversals.

\begin{proposition}
The set
\[
    \Lambda\times L\defeq\set*{(x,a)\;\middle|\;x\in\Lambda,a\in
    L}\subseteq\R^{p+q}
\]
is also a Delone set.
\end{proposition}

\begin{proof}
Choose \( 0<r<R \) satisfying  
\( r<\min\{r_L,r_{\Lambda}\} \) and \( R>\sqrt{R_\Lambda^2+R_L^2} \).
Then for all
\( (x,a)\in\R^m\times\R^n \), we have 
\begin{align*}
\Ball((x,a),r)\subseteq\Ball(x,r_\Lambda)\times\Ball(a,r_L),\quad
\Ball((x,a),R)\supseteq\Ball(x,R_\Lambda)\times\Ball(a,R_L).
\end{align*}
Hence 
\begin{align*}
\#(\Ball((x,a),r)\cap(\Lambda\times
L))&\leq\#((\Ball(x,r_\Lambda)\times\Ball(a,r_L))\cap (\Lambda\times L)) \\
&=\#((\Ball(x,r_\Lambda)\cap\Lambda)\times(\Ball(a,r_L)\cap L)) \\
&=\#(\Ball(x,r_\Lambda)\cap\Lambda)\cdot\#(\Ball(a,r_L)\cap L) \\
&\leq 1,\\
\#(\Ball((x,a),R)\cap(\Lambda\times
L))&\geq\#((\Ball(x,R_\Lambda)\times\Ball(a,R_L))\cap (\Lambda\times L)) \\
&=\#((\Ball(x,R_\Lambda)\cap\Lambda)\times(\Ball(a,R_L)\cap L)) \\
&=\#(\Ball(x,R_\Lambda)\cap\Lambda)\cdot\#(\Ball(a,R_L)\cap L) \\
&\geq 1.
\end{align*}
So \( \Lambda\times L \) is an \( (r,R) \)-Delone set.
\end{proof}

A convenient way to define the ``stacking'' map is by describing the
\Cst-algebra of the product Delone set as a spatial tensor product. We
note the following well-known lemma:

\begin{lemma}\label{lem:Cst-algebra_product_groupoid}
Let \( \mathcal{G}_1\rightrightarrows\Omega_1 \) 
and \( \mathcal{G}_2\rightrightarrows\Omega_2 \) 
be \'etale groupoids. Let \( \mathcal{G}_1\times\mathcal{G}_2 \) 
be the product groupoid of \( \mathcal{G}_1 \) and \( \mathcal{G}_2 \). 
Then there is an isomorphism
\[
    \Phi\colon\Cst(\mathcal{G}_1)\otimes\Cst(\mathcal{G}_2)\xrightarrow{\sim}\Cst(\mathcal{G}_1\times\mathcal{G}_2),\quad
    \Phi(f_1\otimes f_2)(\gamma_1,\gamma_2)\defeq
    f_1(\gamma_1)f_2(\gamma_2),
\]
where \( \otimes \) refers to the spatial tensor product.
\end{lemma}

\begin{proof}
This lemma is well-known, and we sketch a proof here.

Write \( L^2(\mathcal{G}_1) \), \( L^2(\mathcal{G}_2) \) and \(
L^2(\mathcal{G}_1\times\mathcal{G}_2) \) for the Hilbert \Cst-modules over \(
\Co(\Omega_1) \), \( \Co(\Omega_2) \) and \( \Co(\Omega_1\times\Omega_2)
\) as in \cref{sec:regular_representation}. We claim that \(
L^2(\mathcal{G}_1\times\mathcal{G}_2) \) is
canonically isomorphic to the external tensor product Hilbert \Cst-module
\( L^2(\mathcal{G}_1)\otimes L^2(\mathcal{G}_2) \) 
under the identification
\[
  \Psi\colon
  \Co(\Omega_1)\otimes\Co(\Omega_2)\simeq\Co(\Omega_1\times\Omega_2),\quad
  \Psi(f\otimes g)(x,y)\defeq f(x)\cdot g(y).
\] 

To see this, let \( f_1,f_2\in\Cc(\mathcal{G}_1) \) and 
\( g_1,g_2\in\Cc(\mathcal{G}_2) \),
Then their tensor products \( f_1\otimes g_1,
f_2\otimes g_2 \) belong to \( \Cc(\mathcal{G}_1\times\mathcal{G}_2) \). 
Given \( x\in\Omega_1 \) and \(
y\in\Omega_2 \), we have
\begin{gather*} 
    \braket*{f_1,f_2}(x)=\sum_{\gamma\in
    r^{-1}(x)}f_1(\gamma)f_2(\gamma),\quad
    \braket*{g_1,g_2}(y)=\sum_{\eta\in r^{-1}(y)}g_1(\eta) g_2(\eta),
\end{gather*}
and
\begin{align*}
\braket*{f_1\otimes g_1,f_2\otimes g_2}(x,y)&=\sum_{(\gamma,\eta)\in
    r^{-1}(x,y)}(f_1\otimes g_1)(\gamma,\eta)(f_2\otimes
    g_2)(\gamma,\eta)\\
    &=\sum_{\gamma\in r^{-1}(x)}\sum_{\eta\in
    r^{-1}(y)}f_1(\gamma)g_2(\gamma)g_1(\eta)g_2(\eta) \\
    &=\braket{f_1,f_2}(x)\cdot\braket{g_1,g_2}(y).
\end{align*}

Thus \( L^2(\mathcal{G}_1\times \mathcal{G}_2)\simeq
L^2(\mathcal{G}_1)\otimes L^2(\mathcal{G}_2) \) as Hilbert \Cst-modules
over \( \Co(\Omega_1\times\Omega_2)\simeq\Co(\Omega_1)\otimes\Co(\Omega_2)
\). This induces an injective 
\st-homomorphism (cf.~\citelist{\cite{Blackadar:K-theory}*{Section
13.5} \cite{Lance:Hilbert_Cst-modules}*{Chapter 4}}):
\[ 
  \Phi\colon\Bdd_{\Co(\Omega_1)}(L^2(\mathcal{G}_1))\otimes\Bdd_{\Co(\Omega_2)}(L^2(\mathcal{G}_2))\to\Bdd_{\Co(\Omega_1\times\Omega_2)}(L^2(\mathcal{G}_1\times\mathcal{G}_2)).
\]
In particular, for 
\( f_1\in\Cc(\mathcal{G}_1) \) and \( f_2\in\Cc(\mathcal{G}_2) \),
viewed as convolution 
operators on \( L^2(\mathcal{G}_1) \) and \( L^2(\mathcal{G}_2)
\) via their regular representations \( \pi_i\colon
\Cst(\mathcal{G}_i)\to\Bdd_{\Co(\Omega_i)}(\mathcal{G}_i) \) for \( i=1,2
\), then \( \Phi \) sends \( f_1\otimes f_2 \) to the
convolution operator \( f_1\otimes f_2\in\Cc(\mathcal{G}_1\times
\mathcal{G}_2) \) on \( L^2(\mathcal{G}_1\times \mathcal{G}_2) \). 
Passing to the norm closure, we find that \( \Phi \) maps \(
\Cst(\mathcal{G}_1)\otimes\Cst(\mathcal{G}_2) \) into \(
\Cst(\mathcal{G}_1\times\mathcal{G}_2) \), and has dense range as \(
\Cc(\mathcal{G}_1\times\mathcal{G}_2) \) is dense in \(
\Cst(\mathcal{G}_1\times \mathcal{G}_2) \). Thus \( \Phi \) is both
injective and has dense range, hence a \st-isomorphism
\[ 
  \Cst(\mathcal{G}_1)\otimes\Cst(\mathcal{G}_2)\xrightarrow{\sim}\Cst(\mathcal{G}_1\times \mathcal{G}_2).\qedhere
\]
\end{proof}

\begin{corollary}\label{cor:Cst-algebra_Lambda_times_L}
There is a canonical isomorphism
\[ 
    \Cst(\mathcal{G}_{\Lambda\times
    L})\simeq\Cst(\mathcal{G}_\Lambda)\otimes\Cst(\mathcal{G}_L).
\]
\end{corollary}
\begin{proof}
By \cref{lem:Cst-algebra_product_groupoid}, it suffices to prove that the
\'etale groupoids \(
\mathcal{G}_{\Lambda\times L} \) and 
\( \mathcal{G}_{\Lambda}\times\mathcal{G}_L \) are isomorphic. 

We claim that \( \Omega_{\Lambda\times L}\simeq\Omega_\Lambda\times\Omega_L
\) are isomorphic as \( \R^{m+n} \)-spaces. We have
\[
    \Lambda\times
    L\in\Del_{(r_\Lambda,R_\Lambda)}(\R^m)\times\Del_{(r_L,R_L)}(\R^n).
\]
The translation action of \( \R^{p+q}=\R^p\times\R^q \) preserves \(
\Del_{(r_\Lambda,R_\Lambda)}(\R^p)\times\Del_{(r_L,R_L)}(\R^q) \). 
By \cref{prop:spaces_of_measures}, both \(
\Del_{(r_\Lambda,R_\Lambda)}(\R^m)\times\Del_{(r_L,R_L)}(\R^n) \) 
are closed,
hence their product is a closed subset of \( \mathcal{M}(\R^{p+q}) \).
Since the product of the orbits of \( \Lambda \) and \( L \) is contained
in \( \Del_{(r_\Lambda,R_\Lambda)}(\R^p)\times\Del_{(r_L,R_L)}(\R^q) \), we
conclude that \( \Omega_{\Lambda\times L} \), \(
\Omega_\Lambda\times\Omega_L \) are contained in it. 
Hence, \( \Omega_{\Lambda\times L}\subseteq\Omega_\Lambda\times\Omega_L
\) as it is the smallest closed subset in \(
\Del_{(r_\Lambda,R_\Lambda)}(\R^p)\times\Del_{(r_L,R_L)}(\R^q) \), which 
contains the product of the joint orbits of \( \Lambda \) and \( L \).

We claim that \( \Omega_{\Lambda\times
L}\supseteq\Omega_\Lambda\times\Omega_L \) as well. Let \(
(\omega,\ell)\in\Omega_\Lambda\times\Omega_L \), then there exists nets \(
(x_\alpha)_{\alpha\in A}\subseteq\R^p \) and \( (a_\beta)_{\beta\in
B}\subseteq\R^q \) such that 
\[ \Lambda+x_\alpha\to\omega\quad\text{and}\quad L+a_\beta\to\ell \quad\text{in the
weak\st-topology.} \] 
Then the net \( (x_\alpha,a_\beta)_{\alpha\times\beta\in A\times B} \),
where \( A\times B \) carries the lexicographic order, satisfies  
\[
    (\Lambda+x_\alpha,L+a_\beta)_{(\alpha,\beta)\in A\times
    B}\to(\omega,\ell)\quad\text{in the weak\st-topology}.
\]
Then we conclude that \( \Omega_{\Lambda\times
L}\simeq\Omega_\Lambda\times\Omega_L \) as \(
\R^{p+q} \)-spaces.

As a consequence, the action groupoids \( \Omega_{\Lambda\times
L}\rtimes\R^{p+q} \) and \( (\Omega_{\Lambda}\times\Omega_L)\rtimes\R^{p+q}
\) are isomorphic topological groupoids, and have homeomorphic abstract
transversals \( \Omega_0^{\Lambda\times L} \) and \(
\Omega_0^{\Lambda}\times\Omega_0^L \). This implies an isomorphism of 
\'etale groupoids
\[
\mathcal{G}_{\Lambda\times L}\simeq\mathcal{G}_\Lambda\times
\mathcal{G}_L,\quad (\mu,z)\mapsto((\mu_\Lambda,x),(\mu_L,a))
\]
where \( \mu_\Lambda \) and \( \mu_L \)  are 
the images of \( \mu \) under the coordinate
projections \( \R^p\times\R^q\to\R^p \) and \( \R^p\times\R^q\to\R^q \). 
This finishes the proof.
\end{proof}

\cref{cor:Cst-algebra_Lambda_times_L} allows us to define a
\st-homomorphism between groupoid \Cst-algebras.  Since \( \mathcal{G}_L \)
is an \'etale groupoid with compact unit space, its \Cst-algebra \(
\Cst(\mathcal{G}_L) \) has a unit, given by the constant function \(
1_{\Omega_0^L} \) on the unit space.  Define
\begin{equation}\label{eq:phi^Gpd}
\varphi^\Gpd\colon\Cst(\mathcal{G}_\Lambda)\to\Cst(\mathcal{G}_\Lambda)\otimes\Cst(\mathcal{G}_L),\quad
\varphi^\Gpd(f)\defeq f\otimes 1_{\Omega_0^L}.  \end{equation} Then \(
\varphi^\Gpd \) gives a \st-homomorphism \(
\Cst(\mathcal{G}_\Lambda)\to\Cst(\mathcal{G}_{\Lambda\times L}) \) under
the isomorphism \( \Cst(\mathcal{G}_{\Lambda\times
L})\simeq\Cst(\mathcal{G}_\Lambda)\otimes\Cst(\mathcal{G}_L) \) in
\cref{cor:Cst-algebra_Lambda_times_L}. We write \(
\varphi^\Gpd_N\colon\Mat_N(\Cst(\mathcal{G}_\Lambda))\to\Mat_N(\Cst(\mathcal{G}_{\Lambda\times
L})) \) for its entrywise extension to matrix algebras.

Now we pass to Roe \Cst-algebras.
Let \( \omega\in\Omega_0^\Lambda \) and \( \ell\in\Omega_0^L \).
Define their Roe \Cst-algebras \( \CstRoe(\omega) \) and \(
\CstRoe(\omega\times\ell) \) using their standard ample modules \(
\ell^2(\omega,\mathcal{K}) \) and \( \ell^2(\omega\times\ell,\mathcal{K})
\). Let
\begin{equation}\label{eq:phi^Roe}
\varphi^\Roe\colon\Bdd(\ell^2(\omega,\mathcal{K}))\to\Bdd(\ell^2(\omega\times\ell,\mathcal{K})),\quad
T\mapsto T\otimes\id_{\ell^2(\ell)}.
\end{equation}
Then
\[ 
    \braketvert*{x,a}{\varphi^\Roe(T)}{y,b}
    =\braketvert*{x}{T}{y}\cdot\braket*{a\mid b}=T_{x,y}\cdot\delta_{a,b}.
\]
So \( \varphi^\Roe(T) \) is locally compact or controlled if and only if \( T \) is
locally compact or controlled
(cf.~\cref{ex:discrete_space_locally_compact_controlled}). Thus \(
\varphi^\Roe \) maps \( \CstRoe(\omega) \) into \(
\CstRoe(\omega\times\ell) \).     

\begin{lemma}\label{lem:flasque_map}
The map \( \varphi^\Roe\colon\CstRoe(\omega)\to\CstRoe(\omega\times\ell) \)
induces zero maps
\[
    \varphi^\Roe_*\colon\KK(\Cl_{n,0},\CstRoe(\omega))\to\KK(\Cl_{n,0},\CstRoe(\omega\times\ell))
\]
for all \( n \).
\end{lemma}

\begin{proof}
Let
\[ 
    \ell_+\defeq\ell\cap(\R^{q-1}\times\R_{\geq 0}),\quad
    \ell_-\defeq\ell\cap(\R^{q-1}\times\R_{< 0}),
\]
then there are isometries
\[
    V_\pm\colon\ell^2(\omega\times\ell_\pm,\mathcal{K})\to\ell^2(\omega\times\ell,\mathcal{K}),
\]
which induce a diagonal embedding
\begin{align*} 
\Bdd(\ell^2(\omega\times\ell_+,\mathcal{K}))\oplus\Bdd(\ell^2(\omega\times\ell_-,\mathcal{K}))&\to\Bdd(\ell^2(\omega\times\ell,\mathcal{K})),\\
(T_+,T_-)&\mapsto \Ad_{V_+}T_++\Ad_{V_-}T_-.
\end{align*}
Let \( T\in\Bdd(\ell^2(\omega,\mathcal{K})) \). Then \(
T\otimes\id_{\ell^2(\ell)}=\Ad_{V_+}(T\otimes\id_{\ell^2(\ell_+)})
+\Ad_{V_-}(T\otimes\id_{\ell^2(\ell_-)}) \). 
Therefore, the map \( \varphi^\Roe \) agrees with the following
composition:
\begin{multline*} 
    \Bdd(\ell^2(\omega,\mathcal{K}))\to\Bdd(\ell^2(\omega\times\ell_+,\mathcal{K}))\oplus\Bdd(\ell^2(\omega\times\ell_-,\mathcal{K}))\to\Bdd(\ell^2(\omega\times\ell,\mathcal{K})),\\
    T\mapsto
    \left(T\otimes\id_{\ell^2(\ell_+)},T\otimes\id_{\ell^2(\ell_-)}\right)
    \mapsto\Ad_{V_+}\left(T\otimes\id_{\ell^2(\ell_+)}\right)+\Ad_{V_-}\left(T\otimes\id_{\ell^2(\ell_-)}\right).
\end{multline*}
A similar argument as the lines below \eqref{eq:phi^Roe} shows that \(
T\mapsto(T\otimes\id_{\ell^2(\ell_+)},T\otimes\id_{\ell^2(\ell_-)}) \) maps
\( \CstRoe(\omega) \) into \(
\CstRoe(\omega\times\ell_+)\oplus\CstRoe(\omega\times\ell_-) \). 
Both spaces \( \omega\times\ell_\pm \) 
are flasque (cf.~\cite{Roe:Index_theory_coarse_geometry}*{Definition 9.3}).
Hence, \( \CstRoe(\omega\times\ell_\pm) \) have vanishing K-theory by an
Eilenberg swindle argument,
cf.~\cite{Roe:Index_theory_coarse_geometry}*{Theorem 9.4}. Thus \(
\varphi^\Roe_* \) vanishes as it factors through zero.
\end{proof}

\begin{theorem}[Stacked topological phases are
weak]\label{thm:stacked_phases_are_weak}
Let \( \xi^\Gpd_{\omega\times\ell,N} \) be the spectral triple over \(
\Cst(\mathcal{G}_{\Lambda\times L})\otimes\Mat_N(\R) \) defined in
\eqref{eq:groupoid_position_spectral_triple}. Then for each \( n \),
its induced map
\[ 
    \KK(\Cl_{n,0},\Cst(\mathcal{G}_{\Lambda\times L}))\to\KK(\Cl_{n,p+q},\R)
\]
vanishes on the image of
\[
    \varphi^\Gpd_*\colon\KK(\Cl_{n},\Cst(\mathcal{G}_\Lambda))\to
    \KK(\Cl_{n},\Cst(\mathcal{G}_{\Lambda\times L})).
\]
\end{theorem}
\begin{proof}
We claim that the following diagram commutes:
\[ \begin{tikzcd}
\Mat_N\Cst(\mathcal{G}_\Lambda) 
\arrow[r,"\varphi^\Gpd_N"] \arrow[d,"\pi^N_\omega"] &
\Mat_N\Cst(\mathcal{G}_{\Lambda\times L}) 
\arrow[d,"\pi^N_{\omega\times\ell}"] \\
\CstRoe(\omega) \arrow[r,"\varphi^\Roe"] &
\CstRoe(\omega\times\ell).
\end{tikzcd} \]
To see this, let \( f\in\Cc(\mathcal{G}_\Lambda) \) and \( S\in\Mat_N(\R)
\). It follows from \eqref{eq:phi^Gpd} and
\eqref{eq:phi^Roe} that:
\begin{align*} 
\braketvert*{x,a}{\varphi^\Roe\pi_\omega^N(f\otimes
S)}{y,b}&=\braketvert*{x}{\pi_\omega(f)}{y}\cdot\delta_{a,b}\cdot S\\
&=f(\omega-x,y-x)\cdot\delta_{a,b}\cdot S;\\
\braketvert*{x,a}{\pi_{\omega\times\ell}^N\varphi^\Gpd(f\otimes
S)}{y,b}&=\varphi^\Gpd(f)(\omega\times\ell-(x,a),(y-x,b-a))\cdot S\\
&=f(\omega-x,y-x)\cdot 1_{\Omega_0^L}(\ell-a,b-a)\cdot S \\
&=f(\omega-x,y-x)\cdot\delta_{a,b}\cdot S
\end{align*}
holds for all \( x,y\in\omega \) and \( a,b\in\ell \). Therefore the maps \(
\varphi^\Roe\circ\pi_\omega^N \) coincides with \(
\pi_{\omega\times\ell}^N\circ\varphi^\Gpd \) on all elements of the form \(
f\otimes S \), and hence for \(
\Cst(\mathcal{G}_\Lambda)\otimes\Mat_N(\R) \). That is, the diagram above
commutes.

The commutative diagram above gives a commutative diagram in K-theory.
This, together with the bottom left triangle of
\eqref{eq:position_detects_strong_phases}, gives the following commutative
diagram:
\[ 
\begin{tikzcd}
\KK(\Cl_{n,0},\Cst(\mathcal{G}_\Lambda)) \arrow[r,"\varphi^\Gpd_*"] \arrow[d,"\pi_\omega"] &
\KK(\Cl_{n,0},\Cst(\mathcal{G}_{\Lambda\times L})) \arrow[d,"\pi_{\omega\times\ell}"]
\arrow[r,"\xi^\Gpd_{\omega\times\ell}"] & \KK(\Cl_{n,p+q},\R) \\
\KK(\Cl_{n,0},\CstRoe(\omega)) \arrow[r,"\varphi^\Roe_*"] &
\KK(\Cl_{n,0},\CstRoe(\omega\times\ell)) \arrow[ru,"\xi^\Roe_{\omega\times\ell}"'] &
\end{tikzcd}
\]
where \( \xi^\Gpd_{\omega\times\ell} \) and \( \xi^\Roe_{\omega\times\ell}
\) are the corresponding position spectral triples over \(
\Cst(\mathcal{G}_{\Lambda\times L}) \) and \( \CstRoe(\omega\times\ell) \).
It follows that the composition map \(
\xi^\Gpd_{\omega\times\ell}\circ\varphi_*
=\xi^\Roe_{\omega\times\ell}\circ\varphi^\Roe_*\circ\pi_\omega
\)
factors through
\[ 
    \varphi^\Roe_*\colon\KK(\Cl_{n,0},\CstRoe(\omega))\to\KK(\Cl_{n,0},\CstRoe(\omega\times\ell)),
\]
which vanishes by \cref{lem:flasque_map}. Therefore \(
\xi^\Gpd_{\omega\times\ell}\circ\varphi^\Gpd_* \) must be the zero map.
\end{proof}

\begin{example}
Let \( \Lambda=\Z^d \) and \( L=\Z \). Then \( \mathcal{G}_\Lambda
\simeq\Z^d \) and \( \mathcal{G}_L\simeq\Z \) as topological groupoids. 
The unit spaces \( \Omega_0^{\Lambda} \) 
and \( \Omega_0^{\Lambda\times L} \) are both
singletons. So there are unique localised regular representations \(
\pi\colon\Cst(\Z^d)\to\CstRoe(\Z^d) \) and \(
\Cst(\Z^{d+1})\to\CstRoe(\Z^{d+1}) \).
Let \( N=1 \). The \st-homomorphism \(
\varphi^\Gpd_1\colon\Cst(\Z^d)\to\Cst(\Z^{d+1}) \) maps an element in 
\( \Cst(\Z^{d+1}) \) to its restriction to the first \( d \)-coordinates. 
Identify \( \Cst(\Z^d) \) with \( \Cont(\T^d) \) via Fourier transform,
then the map \( \varphi^\Gpd \) sends \( f\in\Cont(\T^d) \) to \( f\otimes
1\in\Cont(\T^d)\otimes\Cont(\T)=\Cont(\T^{d+1}) \). The map \( \varphi^\Gpd \)  
is induced by the injective group homomorphism
\[ 
    \Z^d\to\Z^{d+1},\quad x\mapsto (x,0),
\]
and its image belongs to the kernel of the map
\[
    \KK(\Cl_{n,0},\Cst(\Z^d))\to\KK(\Cl_{n,0},\CstRoe(\Z^d)),
\]
which yields \cite{Ewert-Meyer:Coarse_geometry}*{Proposition 10}. 
\end{example}

\addtocontents{toc}{\SkipTocEntry}
\section*{Acknowledgement} 
The author is grateful to the anonymous referee 
for various illuminating
suggestions, which significantly improved the results of this work.
The author also thanks Chris Bourne, Matthias Frerichs, Matthias
Ludewig and Guo Chuan Thiang for helpful discussions, and
his PhD supervisor Bram Mesland for heartfelt
encouragement and a careful proofreading, helping the author to
eliminate various errors in an earlier version of this work. 
The work is
part of NWO project 613.009.142
``\href{https://www.nwo.nl/projecten/613009142}{Noncommutative index theory
of discrete dynamical systems}''.

\addtocontents{toc}{\SkipTocEntry}
\section*{Author declarations}

\addtocontents{toc}{\SkipTocEntry}
\subsection*{Conflict of interest}
The author has no conflict of interest to disclose.

\addtocontents{toc}{\SkipTocEntry}
\subsection*{Data availability}
Data sharing is not applicable in this study as no new data were created or
analysed.

\bibliography{references}
\end{document}